\documentclass[letterpaper, 11pt,  reqno]{amsart}

\usepackage[margin=1.2in,marginparwidth=1.5cm, marginparsep=0.5cm]{geometry}

\usepackage{amsmath,amssymb,amscd,amsthm,amsxtra, esint, xcolor}

\usepackage[implicit=true]{hyperref}

\allowdisplaybreaks[2]

\usepackage{enumitem}
\usepackage{cancel}

\usepackage[mathscr]{euscript}

\usepackage{color}

\definecolor{gr}{rgb}   {0.,   0.69,   0.23 }
\definecolor{bl}{rgb}   {0.,   0.5,   1. }
\definecolor{mg}{rgb}   {0.85,  0.,    0.85}
\definecolor{yl}{rgb}   {0.8,  0.7,   0.}
\definecolor{or}{rgb}  {0.7,0.2,0.2}

\newtheorem{theorem}{Theorem} [section]

\newtheorem{lemma}[theorem]{Lemma}
\newtheorem{proposition}[theorem]{Proposition}
\newtheorem{remark}[theorem]{Remark}

\newtheorem{definition}[theorem]{Definition}
\newtheorem{corollary}[theorem]{Corollary}

\newtheorem*{ack}{Acknowledgments}

\DeclareMathOperator*{\intt}{\int}

\DeclareMathOperator*{\supp}{supp}

\newcommand{\noi}{\noindent}
\newcommand{\Z}{\mathbb{Z}}
\newcommand{\R}{\mathbb{R}}

\newcommand{\T}{\mathbb{T}}

\let\P= \undefined
\newcommand{\P}{\mathbf{P}}

\newcommand{\M}{\mathcal{M}}

\newcommand{\FL}{\mathcal{F}L} 

\newcommand{\RR}{\mathcal{R}}

\renewcommand{\S}{\mathcal{S}}

\newcommand{\K}{\mathcal{K}}

\newcommand{\F}{\mathcal{F}}

\newcommand{\Nf}{\mathfrak{N}}
\newcommand{\Rf}{\mathfrak{R}}

\newcommand{\al}{\alpha}
\newcommand{\be}{\beta}

\newcommand{\dl}{\delta}

\newcommand{\Dl}{\Delta}
\newcommand{\eps}{\varepsilon}
\newcommand{\kk}{\kappa}
\newcommand{\g}{\gamma}
\newcommand{\G}{\Gamma}

\newcommand{\s}{\sigma}

\newcommand{\ft}{\widehat}

\newcommand{\wt}{\widetilde}
\newcommand{\cj}{\overline}

\newcommand{\dt}{\partial_t}
\newcommand{\dd}{\partial}

\newcommand{\ta}{\theta}

\renewcommand{\l}{\ell}

\newcommand{\les}{\lesssim}
\newcommand{\ges}{\gtrsim}

\newcommand{\jb}[1]
{\langle #1 \rangle}

\newcommand{\ind}{\mathbf{1}}

\newcommand{\NB}{\mathbb{N}}

\newcommand{\NN}{\mathcal{N}}

\newcommand{\uu}{{\bf u}}
\newcommand{\vv}{{\bf v}}

\newcommand{\TT}{\mathcal{T}}
\newcommand{\n}{{\bf n}}

\numberwithin{equation}{section}
\numberwithin{theorem}{section}

\makeatletter
\@namedef{subjclassname@2020}{\textup{2020} Mathematics Subject Classification}
\makeatother

\begin{document}
\baselineskip = 14pt

\title[Hyperbolic NLS]
{Sharp unconditional well-posedness of the 
2-$d$ periodic
cubic hyperbolic nonlinear Schr\"odinger equation}

\author[E.~Ba\c{s}ako\u{g}lu, T.~Oh, and Y.~Wang]
{Engin Ba\c{s}ako\u{g}lu, Tadahiro Oh, and Yuzhao Wang}

\address{Engin Ba\c{s}ako\u{g}lu, 
Institute of Mathematical Sciences, ShanghaiTech University, Shanghai, 201210, China}
\email{ebasakoglu@shanghaitech.edu.cn}


\address{
Tadahiro Oh, 
School of Mathematics\\
The University of Edinburgh\\
and The Maxwell Institute for the Mathematical Sciences\\
James Clerk Maxwell Building\\
The King's Buildings\\
Peter Guthrie Tait Road\\
Edinburgh\\
EH9 3FD\\
 United Kingdom\\
and School of Mathematics and Statistics, Beijing Institute of Technology, Beijing 100081, China}

\email{hiro.oh@ed.ac.uk}

\address{Yuzhao Wang, School of Mathematics, University of Birmingham, Watson Building, Edgbaston, Birmingham B15 2TT, United Kingdom}
\email{y.wang.14@bham.ac.uk}



\begin{abstract}

We study semilinear local  well-posedness 
of the two-dimensional periodic cubic hyperbolic nonlinear Schr\"odinger equation (HNLS) 
in Fourier--Lebesgue spaces.
By employing 
 the Fourier restriction norm method, 
we first establish sharp semilinear local well-posedness of HNLS
in  Fourier--Lebesgue spaces (modulo the endpoint case), including
almost scaling-critical Fourier--Lebesgue spaces.
Then, by adapting the normal form approach, developed
by the second author with Guo and Kwon (2013)
and 
by the second and third authors (2021), to the current hyperbolic setting,  
we  establish 
sharp 
unconditional uniqueness of HNLS
within the 
semilinear local well-posedness regime.
As a key ingredient to both results, 
we establish sharp counting estimates 
for the hyperbolic Schr\"odinger equation. 
As a byproduct of our analysis, 
we also obtain 
sharp 
unconditional uniqueness of the (usual) 
two-dimensional periodic cubic nonlinear Schr\"odinger equation
in Fourier--Lebesgue spaces for $p \ge 3$.

\end{abstract}


%
\maketitle

\vspace*{-15mm}

\tableofcontents

\section{Introduction}
\subsection{Hyperbolic nonlinear Schr\"odinger equation}

We consider the following cubic hyperbolic nonlinear Schr\"odinger equation
(HNLS) on $\T^2 = (\R/2\pi \Z)^2$:\footnote{By convention, we endow
$\T^2$ with the normalized Lebesgue measure $ dx_{\T^2} =  (2\pi)^{-2}dx$
such that we do not need to carry factors involving $2\pi$.}
\begin{equation}\label{HNLS1}
i\dt u+\Box u+  |u|^{2}u=0, 
\end{equation}

\noi
where  the hyperbolic Laplacian $\Box$ is given by 
$$
\Box = \Box_x = \dd_{x_1}^2-\partial_{x_2}^2
\quad \text{with } \ \ x =(x_1, x_2). 
$$

\noi
HNLS \eqref{HNLS1} appears  in the study of modulation of wave trains in gravity water waves;
(see, for example, \cite{Totz1, Totz2}), 
and its well-posedness issue and 
related problems have been studied extensively;
see
 \cite{KN, GT2012,godet2013lower, Wang, MT2015,BD2017,
DMPS, 
taira2020, saut2024, BSTW}.
 In \cite{Wang}, 
the third author proved sharp local well-posedness 
of the cubic HNLS \eqref{HNLS1} in $H^s(\T^2)$, $s > \frac12$
(missing the endpoint $s = \frac 12$).
In this paper, 
we extend this (essentially) sharp local well-posedness
result \cite{Wang}
on the $L^2$-based Sobolev spaces
to the more general Fourier--Lebesgue spaces
(Theorem \ref{THM:1}).
Here, given $s \in \R$ and $1 \le p \le \infty$, the Fourier--Lebesgue space $\F L^{s, p}(\T^2)$
is defined by the norm:
\begin{align*}
\| f\|_{\mathcal{F}L^{s,p}(\T^2)}=
\| \jb{n}^s \ft f(n) \|_{\l^p(\Z^2)}, 
\end{align*}

\noi
where $\jb{\,\cdot\,} = (1+|\,\cdot\,|^2)^\frac{1}{2}$, 
such that, when $ p = 2$, we have $\F L^{s, 2}(\T^2) = H^s(\T^2)$.
When $s= 0$, we set $\FL^p(\T^2) = \FL^{0, p}(\T^2)$.
Furthermore, we 
establish unconditional uniqueness
of solutions to \eqref{HNLS1}
by adapting the normal form approach
developed in~\cite{GKO, OW2}
to the current hyperbolic setting;
see Theorem~\ref{THM:2}.
We also establish sharp unconditional semilinear local well-posedness
of the associated normal form equation (see~\eqref{NF1a} below)
in the Fourier--Lebesgue spaces (modulo the endpoint cases); see Theorem~\ref{THM:3}.

\begin{remark}\label{REM:p1} \rm
When $p = 1$, the Fourier--Lebesgue space $\FL^{1}(\T^2)$ corresponds 
to the Wiener algebra.
As a result,  unconditional local well-posedness
of  \eqref{HNLS1} in $\F L^{s, 1}(\T^2)$ for $s\ge 0$
trivially follows from a contraction argument, using 
the unitarity of the hyperbolic Schr\"odinger propagator $e^{it \Box}$ in 
$\FL^{s, 1}(\T^2)$ 
and the algebra property of $\FL^{s, 1}(\T^2)$, 
without making use of any multilinear dispersion.
For this reason, 
we restrict our attention to $p > 1$
in the remaining part of this paper.
See also Remark \ref{REM:p12}.
\end{remark}

\medskip

Let us first consider the following (usual) cubic nonlinear Schr\"odinger equation on $\T^2$:
\begin{align}
i \partial_t u + \Delta u \pm  |u|^{2} u = 0, 
\label{NLS1}
\end{align}

\noi
where $\Dl = \dd_{x_1}^2 + \dd_{x_2}^2 $
denotes the usual (elliptic) Laplacian.
For clarity, we refer to \eqref{NLS1} as the elliptic NLS
in the following.
Both the cubic HNLS \eqref{HNLS1} and the cubic elliptic NLS \eqref{NLS1} on $\T^2$
enjoy the scaling symmetry,\footnote{In the current periodic setting, 
the scaling symmetry dilates the spatial domain but we ignore this issue here.} 
which induces the following scaling critical regularity
with respect to the Fourier--Lebesgue space $\F L^{s, p}(\T^2)$:
\begin{align*}
s_\text{crit}(p) = 1 - \frac 2p.
\end{align*}

\noi
In particular, when $p = 2$, we have $s_\text{crit}(2) = 0$.
In \cite{BO93}, Bourgain introduced the Fourier restriction norm method, 
based on the $X^{s, b}$-space (see \eqref{Xsb1} with $p = 2$), 
and 
proved {\it semilinear}\footnote{Namely, 
local well-posedness follows from a contraction argument,
which yields a smooth solution map.} 
local well-posedness of the cubic elliptic NLS \eqref{NLS1}
in $H^s(\T^2)$ for any $s > 0$.
A key ingredient 
is the following 
$L^4$-Strichartz estimate (with a slight derivative loss): 
\begin{align}
\| e^{it \Dl}  \P_N f\|_{L^4([0, 1]; L^4(\T^2))} \les N^s \|\P_N f\|_{L^2}
\label{Str1}
\end{align}

\noi
for any $s> 0$, 
where  $\P_N$ denotes  the frequency projector
 onto the frequencies $\{|n|\le N\}$.
We point out that 
this local well-posedness result is essentially sharp
in the sense that~\eqref{NLS1} is known to be ill-posed in $H^s(\T^2)$
for $s < 0$
(see~\cite{OH17, Kishi19}).
Moreover, 
when $s = 0$, 
\eqref{Str1}
fails  and  the cubic elliptic NLS \eqref{NLS1}
is semilinearly ill-posed in $L^2(\T^2)$; see \cite{BO93, TTz, Oh13, Kishi14}.
In fact, 
the well-posedness issue of the cubic elliptic NLS \eqref{NLS1} 
in the critical space $L^2(\T^2)$ remains
a challenging open problem.
See \cite[Remark~4.3]{Oh13} and \cite[Section~3]{Kishi14}
(see also~\cite{TTz})
for a sharp lower bound
and  a recent breakthrough work~\cite{HK1} by Herr and Kwak for a sharp upper bound
for the $L^4$-Strichartz estimate on $\T^2$.
We also mention another  recent 
breakthrough work~\cite{HK2} by Herr and Kwak
on global well-posedness of 
the cubic elliptic  NLS~\eqref{NLS1} 
in $H^s(\T^2)$, $s > 0$, for initial data of arbitrary size
in the defocusing case (with the $-$ sign in \eqref{NLS1})
and 
 below the ground state threshold in the focusing case
 (with the $+$ sign in \eqref{NLS1}).

Given $\mu \in \Z$ and $n \in \Z^2$, define 
$ \G_{\mu}^\pm (n)$ by 
\begin{align}
\label{GG0}
\G_{\mu}^\pm (n)=
\big\{({ n}_1,{  n}_2,{ n}_3) \in\G(n): |{  n}|^2_\pm -|{ n}_1|^2_\pm +|{ n}_2|^2_\pm -|{ n}_3|^2_\pm =\mu\big\}, 
\end{align}

\noi
where the elliptic\,/\,hyperbolic modulus of a vector ${n} = (j,k) \in \Z^2$ is defined as
\begin{align}
|{  n} |_\pm^2 =  j^2 \pm k^2
\label{GG0a}
\end{align}

\noi
and $\G( n)$ is given by 
\begin{align}
\G(n)=\big\{(n_1,n_2, n_3) \in(\Z^2)^3: {  n}_1-{  n}_2+{  n}_3= {  n}, 
\ n \ne n_1, n_3\big\}.
\label{GG0b}
\end{align}

\noi
When $\mu = 0$, 
$ \G_0^\pm (n)$ corresponds to the (non-trivial) resonant frequencies
for \eqref{NLS1} and \eqref{HNLS1}
and we set 
\begin{align}
\label{GG1}
    \G_{\text{res}}^\pm ({ n})=
 \G_0^\pm (n).
\end{align}

\noi
Then, the proof of the $L^4$-Strichartz estimate \eqref{Str1}
reduces to bounding the size of 
$\G_{\mu}^+ (n)$, uniformly in $\mu \in \Z$ and $n \in \Z^2$, 
which follows from the divisor counting on the Gaussian integers $\Z[i]\cong \Z^2$.

Let us turn our attention to the cubic HNLS \eqref{HNLS1}.
In this case, 
the main new difficulty comes from the hyperbolic nature of the problem. 
Namely, in the hyperbolic case, the hyperbolic modulus $|n|_-$ in \eqref{GG0a}
is not sign-definite, which makes
the resonant set $\G_\text{res}^-(n)$ much larger.
Nonetheless, in \cite{Wang}, the third author 
 proved  semilinear local well-posedness of \eqref{HNLS1} in $H^s(\T^2)$ for $s>\frac{1}{2}$ 
 by establishing the following {\it sharp} $L^4$-Strichartz estimate 
 for  the hyperbolic Schr\"odinger equation:
 \begin{align*}
\| e^{it\Box}\P_N f\|_{L^4_{t,x}(\T^3)}\lesssim N^{\frac{1}{4}}\| \P_N f \|_{L^2}, 
 \end{align*}

\noi
where  the $\frac 14$-derivative loss is necessary, which is much worse
as compared to the elliptic case~\eqref{Str1}.
By exploiting the hyperbolic nature of the resonant set $\G_{\rm res}^-(n)$ in \eqref{GG1},
he also showed that 
the cubic HNLS~\eqref{HNLS1} on $\T^2$ is semilinearly
ill-posed in $H^s(\T^2)$ for $s < \frac 12$.
This is in sharp contrast with the corresponding problem on 
$\R^2$ and on the product space $\R\times \T$, 
where \eqref{HNLS1} on $\R^2$ 
(and $\R\times \T$, respectively)
is semilinearly locally well-posed in $H^s$
for $s \ge 0$ (and $s > 0$, respectively), 
essentially matching the elliptic NLS well-posedness theory
and the prediction given by the scaling heuristics (recall $s_\text{crit}(2) = 0$); 
see \cite{BSTW}
for a further discussion.

We now state our first result on 
semilinear local well-posedness of \eqref{NLS1} in 
the Fourier--Lebesgue spaces.

 \begin{theorem}\label{THM:1}

 Let $1<  p<\infty$.
 
 \smallskip

\noi
\textup{(i)}
Let    $s>1-\frac{1}{p}$.
Then,   
  the cubic HNLS \eqref{HNLS1} on $\T^2$ is semilinearly locally well-posed in $\mathcal{F}L^{s,p}(\T^2)$.

\smallskip

\noi
\textup{(ii)}
Let 
 $s< 1-\frac{1}{p}$.
Then,   
  the cubic HNLS \eqref{HNLS1} on $\T^2$  is semilinearly ill-posed in $\mathcal{F}L^{s,p}(\T^2)$. 
More specifically,      
 the solution map for \eqref{HNLS1} on $\F L^{s, p}(\T^2)$ fails to be  $C^3$
 for $s< 1-\frac{1}{p}$.
\end{theorem}

Theorem \ref{THM:1} extends
the well-\,/\,ill-posedness result of \eqref{HNLS1} in \cite{Wang}
from the $L^2$-based Sobolev space $H^s(\T^2)$
to the general Fourier--Lebesgue setting, 
yielding sharp semilinear 
well-\,/\,ill-posedness results
(modulo the endpoint case).
When $p = 2$ (already studied in~\cite{Wang}), there is a regularity gap by $\frac 12$ to the scaling-critical regularity
$s_\text{crit}(2) = 0$.
However, by taking $p \to \infty$, 
Theorem \ref{THM:1}\,(i)
yields local well-posedness of \eqref{HNLS1}
in almost critical Fourier--Lebesgue spaces
in the sense similar to \cite{GV}.
We also note that, for $1 <  p < \infty$ and $s > 1 - \frac 1p$, 
we have 
\begin{align}
 \F L^{s, p}(\T^2) \subset L^2(\T^2).
 \label{crit2}
\end{align}

We prove Theorem \ref{THM:1}\,(i) 
by employing the Fourier-restriction norm method adapted to the Fourier--Lebesgue spaces
\cite{GH, FOW, C1, C2} 
and 
by reducing the matter 
to the hyperbolic counting estimate
(Lemma \ref{LEM:count1}).
As such, the uniqueness of the local-in-time solution
constructed in Theorem \ref{THM:1}\,(i)
holds only in the class:\footnote{In view of the time reversibility of \eqref{HNLS1}, 
we only consider positive times in the remaining part of this paper.}
\begin{align}
C([0,T];\mathcal{F}L^{s,p}(\T^2))\cap X^{s,b}_{p}(T), 
\label{crit3}
\end{align}

\noi
where $T = T(\|u(0)\|_{\F L^{s, p}}) > 0$ denotes the local existence time
and 
$X^{s,b}_{p}(T)$ is as in \eqref{Xsb2} (with some appropriate $b$).
Theorem \ref{THM:1}\,(ii)  follows
from a straightforward modification of the argument in \cite{Wang}.
We present a proof of Theorem \ref{THM:1} in Appendix \ref{SEC:A}.

\begin{remark}\label{REM:p12}\rm
Let  $p = 1$.
Then,  as pointed out in 
Remark \ref{REM:p1}, 
Theorem \ref{THM:1}\,(i) holds for $s \ge 0$
(including the endpoint case $s = 0$).
Similarly, Theorems \ref{THM:2}
and  \ref{THM:3}\,(i) 
hold for $s \ge 0$.
On the other hand, 
Theorem \ref{THM:1}\,(ii)
and~\ref{THM:3}\,(ii)
 hold for $s  < 0$;
 see  Subsection \ref{SUBSEC:A2}.

\end{remark}

\subsection{Unconditional uniqueness}

In \cite{Kato}, Kato introduced the following notion of unconditional uniqueness.
We say that a Cauchy problem is unconditionally (locally) well-posed in 
a given Banach space $B$  if for every initial data
$u_0 \in B$, there exist $T > 0$ and a unique solution $u \in C([0,T];B)$
with $u|_{t = 0} = u_0$, 
where the uniqueness holds in the entire class 
$C([0,T];B)$
without intersecting with any auxiliary function space 
(such as the $X^{s, b}$-spaces).
We refer to such uniqueness 
as unconditional uniqueness. Unconditional uniqueness is a concept of uniqueness which does not depend on how solutions are constructed, 
and thus is of fundamental importance.
As such, 
unconditional uniqueness for dispersive equations
 has been studied extensively
by various methods;
see, for example, 
 \cite{Furioli, yin, GKO, Han, 
 KOPV, Chen, HV, OW2, Kishi21, 
ChenHolmer2,  ChenShenZhang}.

As mentioned above, the construction of solutions
in Theorem \ref{THM:1}\,(i)
relies on the Fourier restriction norm method 
adapted to the Fourier--Lebesgue spaces.
Thus, uniqueness
holds only conditionally, i.e.~in the restricted class \eqref{crit3}.
In this subsection, we first investigate
 unconditional uniqueness for the cubic HNLS \eqref{HNLS1} on $\T^2$.
In order to make sense of the cubic nonlinearity for a fixed time, 
we need $u(t)$ to belong to $L^3(\T^2)$.
Recall  the following  (essentially) sharp embeddings:

\smallskip

\begin{itemize}
\item[(a)]
By Hausdorff-Young's and Sobolev's inequalities, we have 
\begin{align}
\F L^{s, p}(\T^2) \subset W^{s, p'} (\T^2) \subset L^3(\T^2)
\label{emb0}
\end{align}

\noi
for 
\begin{align}
\frac 32 \le p \le 2 \quad \text{and} \quad s \ge \frac 43 - \frac 2p, 
\label{emb0a}
\end{align}

\noi
where  $W^{s, p}(\T^2)$ denotes the $L^p$-based Sobolev space.
 Here, $p'$ denotes the H\"older conjugate of $p$.
\medskip

\item[(b)]
Under the conditions:
\begin{align}
\text{(b.i)} \ \ 1\leq p\leq \frac{3}{2}, \  s\ge 0\qquad \text{or}
\qquad \text{(b.ii)}\ \ p>2, \  s>\frac{4}{3}-\frac{2}{p},
\label{emb2}
\end{align}

\noi 
we have 
\begin{align}
   \mathcal{F}L^{s,p}(\T^2)\subset \mathcal{F}L^{\frac{3}{2}}(\T^2)\subset L^{3}(\T^2), 
\label{emb1}
\end{align}

\noi
where the first embedding in \eqref{emb1} 
under the condition (b.ii) follows from H\"older's inequality.

\end{itemize}


\noi
Thus, 
we see that the  the Fourier--Lebesgue
regularity \eqref{emb0a} or  \eqref{emb2}
is necessary
for establishing unconditional uniqueness
of \eqref{HNLS1}.

We now state our second result
on unconditional uniqueness of \eqref{HNLS1}.

\begin{theorem}\label{THM:2}
Let $s \in \R$ and $1 <  p < \infty$ satisfy
one of the following conditions\textup{:}
\begin{align}
\begin{split}
\textup{(i)}&  \ \ 
s>1-\frac{1}{p}, \ \ 
\, \, \text{if } \ 1 <  p \le 3, \\
\textup{(ii)}&  \ \ 
 s>\frac{4}{3}-\frac{2}{p}, \ \ 
 \text{if } \ 3 \le p < \infty.
\end{split}
\label{emb3}
 \end{align}

\noi
Then, 
the cubic HNLS \eqref{HNLS1} on $\T^2$ is unconditionally locally well-posed
in $\F L^{s, p}(\T^2)$.
\end{theorem}

Theorem \ref{THM:2} establishes the first 
unconditional uniqueness result
for hyperbolic nonlinear Schr\"odinger equations.
By comparing \eqref{emb2} and \eqref{emb3}, 
we see that Theorem \ref{THM:2}
yields  sharp unconditional uniqueness 
of the cubic HNLS \eqref{HNLS1} on $\T^2$, 
when  $p \ge 3$.
On the other hand, 
when $1 \le p \le 3$, 
the regularity condition $s > 1 - \frac 1p$ in \eqref{emb3}
corresponds to 
the sharp regularity threshold
for semilinear local well-posedness of \eqref{HNLS1}
stated in Theorem \ref{THM:1}
(modulo the endpoint $s = 1- \frac 1p$).
Namely, 
the unconditional uniqueness result in Theorem~\ref{THM:2} is {\it sharp}
within the semilinear local well-posedness regime.
We prove Theorem \ref{THM:2}
by adapting 
 an infinite iteration
of Poincar\'e-Dulac normal form reductions, 
originally introduced 
\cite{GKO} by the second author with Guo and Kwon
and further developed by the second and third authors in \cite{OW2}, 
to the current hyperbolic setting.
More precisely, by iteratively applying normal form reductions, 
we transform the cubic HNLS \eqref{HNLS1}
to the normal form equation \eqref{NF1a}
which already encodes multilinear dispersive smoothing.
In the current hyperbolic setting, however, 
the counting estimate for $\G_\mu^-(n)$ in \eqref{GG0}
becomes much worse as compared to the elliptic case, 
which makes our analysis of the multilinear operators, 
appearing in the normal form equation \eqref{NF1a}, 
more involved than those in \cite{GKO, OW2};
see Remark \ref{REM:diff1}.


\begin{remark}\rm 
In  \cite{Kishi21}, 
Kishimoto 
proved  unconditional uniqueness of 
the cubic elliptic NLS~\eqref{NLS1} in $H^s(\T^2)$
for 
$s > \frac 25$.
As we point out in Remark \ref{REM:sharp}, 
 our counting estimate 
for the hyperbolic Schr\"odinger equation
(Lemma~\ref{LEM:count1})
also holds 
for the elliptic  Schr\"odinger equation.
As a consequence, Theorem \ref{THM:2} also holds for the cubic elliptic NLS \eqref{NLS1} on $\T^2$
with the same range,
thus yielding sharp unconditional uniqueness
for  the cubic elliptic NLS~\eqref{NLS1} 
 in $\F L^{s, p}(\T^2)$ for $p \ge 3$.

\end{remark}

\subsection{Normal form approach}

A normal form method 
refers to an approach, where one applies
 a sequence of transforms to reduce dynamics to the essential part, 
corresponding to the (nearly) resonant part.
A general procedure is given as follows:

\smallskip

\begin{itemize}
\item[(1)]
Separate the nonlinear part into (nearly) resonant and (highly) non-resonant parts, 

\smallskip
\item[(2)]
``Eliminate'' the non-resonant part
at the expense of  introducing higher order terms, 

\smallskip
\item[(3)]
Repeat this process indefinitely (or terminate it  at some finite step).

\end{itemize}

\noi
There are different approaches
to 
``eliminate'' the non-resonant part
in the second step.
 The commonly used normal form methods
are the Birkhoff normal form reductions (see, for example,  \cite{BG, CKO}
in the PDE context)
and the Poincar\'e-Dulac normal form reductions.
In this paper, we 
implement an infinite iteration of 
Poincar\'e-Dulac normal form reductions, 
developed in \cite{GKO, OW2}, 
to transform the cubic HNLS \eqref{HNLS1}
into a normal form equation; see~\eqref{NF1a} below.

By writing \eqref{HNLS1} in 
 the Duhamel formulation, we have 
\begin{align}
u(t) = S(t) u(0) + i \int_0^t S(t - t')\big\{  \Nf(u) (t')
+ \Rf(u)(t')\big\}
 dt'
\label{Duha1}
\end{align}	

\noi
where $S(t) = e^{it \Box}$
and we split the nonlinearity $|u|^2 u$ into two parts, 
$\Nf(u) $ and $\Rf(u) $,  
defined by
\begin{align}
\begin{split}
\F_x\big(\Nf(u)\big)(n) 
&  = \sum_{\substack{n = n_1 - n_2 + n_3\\ n\neq n_{1},n_{3}}}
\ft u (n_1) \cj{\ft u (n_2)} \ft u (n_3), \\
\Rf(u) & = |u|^2 u - \Nf(u).
\end{split}
\label{non0}
\end{align}

\noi
Here, $\F_x$ denotes the spatial Fourier transform.
Define the following trilinear operators:
\begin{align}
\begin{split}
\F_x\big(\Nf(u_1,u_2,u_3)\big)(n) 
& = \sum_{\substack{n = n_1 - n_2 + n_3\\ n\neq n_{1},n_{3}}}
\ft u_1 (n_1) \cj{\ft u_2 (n_2)} \ft u_3 (n_3), \\
\F_x\big(\Rf_1(u_1,u_2,u_3)\big) (n) 
& = - 
\ft u_1 (n) \cj{\ft u_2 (n)} \ft u_3 (n), \\
\F_x\big(\Rf_2(u_1,u_2,u_3)\big) (n) 
& = 
 \bigg(\sum_{m \in \Z^2} \ft u_1 (m) \cj{\ft u_2 (m)}\bigg) \ft u_3 (n)\\
& \quad  +  \bigg(\sum_{m \in \Z^2} \ft u_3 (m) \cj{\ft u_2 (m)}\bigg) \ft u_1 (n).
\end{split}
\label{non1}
\end{align}

\noi
With the short-hand notation 
\eqref{short1}, 
we have $\Nf(u) = \Nf(u, u, u)$.
Moreover, from \eqref{non0} and~\eqref{non1}, 
we have
\begin{align}
\begin{split}
|u|^2 u 
& = \Nf(u) + \Rf(u) \\
& = \Nf(u) + \Rf_1(u)  + \Rf_2(u).
\end{split}
\label{non2}
\end{align}

\noi
See also \eqref{tri1a}.

\begin{remark}\rm
Let $\uu$ denote
the interaction representation of $u$, defined
by 
\begin{align*}
\uu (t) = S(-t)u(t).
\end{align*}

\noi
Then, in terms of $\uu$, the cubic HNLS \eqref{HNLS1} is expressed as\footnote{The 
operator $\NN^{(1)}$ is non-autonomous but, 
for simplicity of presentation, we suppress its $t$-dependence.}
\begin{equation}\label{HNLS2}
\begin{aligned}
\partial_t\ft \uu (n)&=i\sum_{\substack{n=n_1-n_2+n_3\\n_2\neq n_1, n_3}}
e^{it\Phi (\bar n) }\widehat{\uu  }(n_1)\overline{\widehat{\uu  }(n_2)}\widehat{\uu  }(n_3)
+ i \F_x \big(\Rf(\uu)\big)(n) \\
&=: \F_x\big(\NN^{(1)}(\uu  )\big)(n)+\F_x\big(\RR^{(1)}(\uu  )\big)(n), 
  \end{aligned}  
\end{equation}

\noi
where $\Rf(\uu)$ is as in \eqref{non0} (see also \eqref{non1} and \eqref{non2})
and 
the modulation function $\Phi(\bar n)$ is given by 
\begin{align}\label{phi1}
\Phi(\bar n)=\Phi(n,n_1,n_2,n_3) = |n|_-^2-|n_1|_-^2+|n_2|_-^2-|n_3|_-^2.
\end{align}

\noi
On the Fourier support of $\RR^{(1)}(\uu)$ (and $\Rf(u)$), we have 
$n = n_1$ or $n_3$, and thus 
$\Phi(\bar n) = 0$.
In the one-dimensional case, 
$\Rf(u)$ indeed corresponds to the resonant part (i.e.~$\Phi(\bar n) = 0$)
of the nonlinearity $|u|^2 u$.
It is, however, not the case in the higher dimensional setting.
For the elliptic cubic NLS
\eqref{NLS1} on $\T^2$, the modulation function $\Phi(\bar n)$
(where we replace the hyperbolic modulus $|\cdot|_-$
in \eqref{phi1}
by the elliptic  modulus $|\cdot|_+$)
can be written as
\begin{align*}
\Phi(\bar n)& = |n|_+^2-|n_1|_+^2+|n_2|_+^2-|n_3|_+^2\\
& = 2 (n - n_1)\cdot (n - n_3).
\end{align*}

\noi
Namely, $\Phi(\bar n) = 0 $ when the vectors
$n - n_1$ and $n - n_3$ are perpendicular.
In view of \eqref{GG0a}, 
the resonant set for the cubic HNLS \eqref{HNLS1}, 
where $\Phi(\bar n) = 0$ vanishes, 
is even larger, which makes the hyperbolic problem much harder.
In the following, 
we refer to  $\Rf(u)$ (and $\RR^{(1)}(\uu)$) 
as  the ``trivial'' resonant part.

\end{remark}

\medskip

Following \cite{GKO, OW2}, 
we implement an infinite iteration of normal form reductions
and derive the following {\it normal form equation}:
 \begin{align}
\begin{split}
u(t) 
 =   S(t) u(0) 
&  +     \sum_{j = 2}^\infty  \Nf_0^{(j)}(u)(t) 
 -  \sum_{j = 2}^\infty \Nf_0^{(j)}(u)(0)  \\
& 
+ \int_0^t S(t - t ') \bigg\{
\sum_{j = 1}^\infty \Nf_1^{(j)}(u)(t')   + \sum_{j = 1}^\infty \Rf^{(j)}(u)(t')\bigg\} dt',
\end{split}
\label{NF1a}
\end{align}

\noi
where $\{\Nf_0^{(j)}\}_{j=2}^{\infty}$ are autonomous  $(2j-1)$-linear operators
while  $\{\Nf_{1}^{(j)}\}_{j=1}^{\infty}$
and $\{\Rf^{(j)}\}_{j=1}^{\infty}$
 are  autonomous  $(2j+1)$-linear operators.
See Section \ref{SEC:NF1} for details.
In particular, see \eqref{NF1b} for the precise definitions
of these multilinear operators.
As compared to the original Duhamel formulation \eqref{Duha1}, 
the normal form equation \eqref{NF1a}
may look more complicated from the algebraic viewpoint.
However, as observed in \cite{GKO, OW2}, the multilinear operators 
$\Nf_0^{(j)}$, $\Nf_1^{(j)}$, and $\Rf^{(j)}$
appearing 
in \eqref{NF1a} already
encode multilinear dispersive smoothing.
As a result, the normal form equation \eqref{NF1a} 
 exhibits better analytical properties than 
the original Duhamel formulation \eqref{Duha1}, 
and 
we can easily prove sharp semilinear local well-posedness
of the normal form equation \eqref{NF1a}
in $\FL^{s, p}(\T)$ by a simple contraction argument
{\it without} any auxiliary function spaces.

\begin{theorem}\label{THM:3}
  Let $1<  p<\infty$.

\smallskip

 \noi
\textup{(i)}
Let  $s>1-\frac{1}{p}$.
Then,   
  the normal form equation \eqref{NF1a} is 
  unconditionally semilinearly locally well-posed in $\mathcal{F}L^{s,p}(\T^2)$.

\smallskip

\noi
\textup{(ii)}
Let  $s< 1-\frac{1}{p}$.
Then,   
  the normal form equation \eqref{NF1a} is semilinearly ill-posed in $\mathcal{F}L^{s,p}(\T^2)$. 
More specifically,      
 the solution map for \eqref{NF1a} on $\F L^{s, p}(\T^2)$ fails to be  $C^3$
 for $s< 1-\frac{1}{p}$.

\end{theorem}

By comparing Theorem \ref{THM:3}\,(i)
with Theorem \ref{THM:2}\,(i), 
we see that 
 the normal form equation~\eqref{NF1a}
behaves better in the low regularity setting than  the cubic HNLS \eqref{HNLS1}.
Hence, the normal form approach can be viewed
as  an analytical  renormalization.
See Remark \ref{REM:UU2}.

Theorem \ref{THM:2}\,(i)
on the unconditional uniqueness of solutions to the original cubic HNLS~\eqref{HNLS1}
follows as an immediate corollary to Theorem \ref{THM:3}, 
once we observe that the $L^3_x$-regularity 
(together with the embeddings \eqref{emb0} and \eqref{emb1})
suffices
to show the equivalence of 
 the original cubic HNLS \eqref{HNLS1} and the normal form equation~\eqref{NF1a}.
Since this part of the argument is standard, we omit details;
see \cite{GKO, OW2}.
See also Remark \ref{REM:UU2}.

In Section \ref{SEC:NF1}, we formally implement
an infinite iteration of normal form reductions
and derive the normal form equation \eqref{NF1a}.
In Section \ref{SEC:NF2}, 
we estimate the relevant multilinear operators.
Once we establish the relevant multilinear estimates
(Lemmas \ref{LEM:N0j}, \ref{LEM:Rj}, 
and~\ref{LEM:N1j} stated in Subsection 
\ref{SEC:NF1d}), 
Theorem \ref{THM:3}\,(i)
on  unconditional local well-posedness
for  the normal form equation \eqref{NF1a}
follows from a simple contraction  argument.
We omit details.
See \cite[Section 2]{OW2}.
As for a proof of 
 Theorem \ref{THM:3}\,(ii), 
 see  
 Subsection \ref{SUBSEC:A2}.
See also Remark \ref{REM:p12}.

\begin{remark}\label{REM:UU2}\rm

As we see in Subsection \ref{SUBSEC:NF1c}
(see Remark \ref{REM:UU}), 
by 
imposing a higher regularity on 
$u \in C([0, T]; \FL^{s, p}(\T^2))$
for $s$ and $p$
satisfying \eqref{emb3}, 
we derive the normal form equation \eqref{NF1a}
in a weaker topology $C([0, T]; \FL^{\infty}(\T^2))$. 
Conversely, under the same regularity assumption~\eqref{emb3}
(in particular, the embedding \eqref{emb0} or \eqref{emb1} holds), 
we can easily show that a solution $u$ to the normal form equation \eqref{NF1a}
also satisfies the cubic HNLS \eqref{HNLS1}.
See \cite[Section 2.2]{OW2} for details.
Namely, under the regularity assumption \eqref{emb3}, 
the cubic HNLS~\eqref{HNLS1} and the normal form equation \eqref{NF1a}
are equivalent.
On the other hand, 
in view of Theorem \ref{THM:3}\,(i), 
we see that when $3 < p < \infty$ and $1 - \frac 1p < s \le \frac 43 - \frac 2p$
 the normal form equation \eqref{NF1a}
behaves better than the original  cubic HNLS~\eqref{HNLS1}.

\end{remark}

\begin{remark}\rm
We also mention a recent preprint
\cite{Bruned} by 
Bruned, 
where the author 
provided a derivation of the normal form equation 
(for the cubic elliptic NLS on $\T$) from the algebraic viewpoint
via arborification.
 See  Remark \ref{REM:Bruned}.
\end{remark}

\medskip

We conclude this subsection 
by stating a corollary to the normal form approach (Theorem~\ref{THM:3})
to the original cubic HNLS \eqref{HNLS1}
(besides Theorem \ref{THM:2}).
Recall the following notion of 
{\it sensible weak solutions};
see \cite{OW1, OW2, FO}.

\begin{definition}\label{DEF:sol}\rm

Let $s \in \R$ and $1<  p < \infty$ and $T>0$. 
Given $u_0 \in \FL^{s, p}(\T^2)$, 
we say that
 $u \in C([0,T]; \FL^{s, p}(\T^2))$
is a sensible weak solution
to the cubic HNLS \eqref{HNLS1}   on $[0, T]$ if,
for any sequence $\{u_{0, m}\}_{m \in \NB}$ of smooth functions
tending to $u_0$ in $\FL^p(\T^2)$, 
the corresponding (classical) solutions $u_m$ with $u_m|_{t = 0} = u_{0, m}$
converge to $u$ 
in $C([0,T]; \F L^{s, p} (\T^2))$.
Moreover, we impose that there exists  a distribution $v$
such that the nonlinearity $|u_m|^2 u_m$ converges to  $v$ in the space-time distributional sense,
independently of the choice of an approximating sequence.

\end{definition}

As pointed out in \cite{OW2}, 
the last part of Definition \ref{DEF:sol}
is not quite necessary, 
since
the convergence of  $u_m$ to $u$ 
in $C([0,T]; \F L^{s, p} (\T^2))$
implies that 
the nonlinearity $|u_m|^2 u_m$  converges to some limit $v$ in the space-time distributional sense.
We, however, keep it for clarity.
We also note that  sensible weak solutions are unique by definition.
See~\cite{OW2} for a further discussion
on various notions of weak solutions.

As a corollary to Theorem \ref{THM:3}, 
we have the following local well-posedness
of the cubic HNLS \eqref{HNLS1}
in the sense of sensible weak solutions.

\begin{corollary}\label{COR:4}
Let $1  < p < \infty$ and  $s>1-\frac{1}{p}$.
Then, the cubic HNLS \eqref{HNLS1} on $\T^2$
 is locally well-posed in $\FL^{s, p}(\T^2)$
in the sense of sensible weak solutions.

\end{corollary}

We note that Corollary \ref{COR:4} follows
from Theorem \ref{THM:1}
obtained via the Fourier restriction norm method.
The point of Corollary \ref{COR:4}
is that it can be obtained solely via the normal form method
without relying on any auxiliary function spaces,
in particular 
 independently of  the Fourier restriction norm method.
The notion of sensible weak solutions
is a standard notion of solutions
in the study of completely integrable PDEs
(see \cite{KT, KV}).
We also point out that an analogous notion of solutions, 
where a solution is defined to be the unique\footnote{Here, 
independence of the choice of an approximating sequence
is required to hold only within a certain class of regularization such as mollification.
See~\cite{OOTz} for a further discussion.} 
limit
of smooth solutions,  
is standard 
in the study of 
singular stochastic PDEs \cite{Hairer, GKO2}.
Hence, 
Corollary \ref{COR:4} via the normal form method
provides a satisfactory solution theory
for the cubic HNLS \eqref{HNLS1} on $\T^2$, 
matching the solution theory via the Fourier restriction norm method (Theorem \ref{THM:1}).
By argument as in 
 \cite[Subsection 2.2]{OW2}, 
Corollary \ref{COR:4}
follows from Theorem \ref{THM:3}
and thus we omit details.

\begin{remark}\label{REM:NF1} \rm

The Poincar\'e-Dulac normal form reductions have been 
used in various contexts:
small data global well-posedness \cite{Shatah}, 
unconditional uniqueness \cite{BIT, KO, GKO, OW2, KishiX, Kishi21,  MP2, FLZ}, 
nonlinear smoothing and the nonlinear Talbot effect \cite{ETz, ETz3, ETz2, CS}, 
existence of global attractors~\cite{ETz3}, 
reducibility~\cite{CGKO}, 
improved energy estimates 
in both the deterministic and probabilistic contexts
\cite{TT, OTz, GO, OW0, OST, OS}.
Recently, it has also been applied to 
modulated dispersive equations
\cite{GLLO}, 
stochastic dispersive equations
\cite{OSW}, 
and  numerical schemes~\cite{CFO}.

\end{remark}

\subsection{Outline of the paper}

In Section \ref{SEC:2}, 
we introduce notations
and establish the hyperbolic counting
estimate (Lemma \ref{LEM:count1}).
In Section \ref{NF1a}, 
we go over a formal derivation of the normal form equation \eqref{NF1a}
and state the relevant multilinear estimates, 
which are then proven in Section \ref{SEC:NF2}.
In Appendix \ref{SEC:A}, we present a proof of Theorem \ref{THM:1}.

\section{Preliminaries}
\label{SEC:2}

In this section, we introduce some notations
and then establish a key hyperbolic counting estimate (Lemma \ref{LEM:count1}).

\subsection{Notations}

Let $A\les B$ denote an estimate of the form $A\leq CB$ for some constant $C>0$. We write $A\sim B$ if $A\les B$ and $B\les A$, while $A\ll B$ denotes $A\leq c B$ for some small constant $c> 0$. 
We use $C>0$ to denote various constants, which may vary line by line.

Given dyadic $N \ge 1$, 
we use the notation $\{|n| \sim N\}$
to denote $\{\frac 12 N \le |n| < 2 N\}$ for $N \ge 2$
and $\{|n| < 2\}$ for $N = 1$ in the following.

Given a multilinear operator $T$, 
we use the following short-hand notation:
\begin{align}
T(f) = T(f, \dots, f).
\label{short1}
\end{align}

\noi
 Furthermore, 
  we use $T(f)(n)$ to denote the Fourier coefficients of $T(f)$.

 We use  
  $ S(t) = e^{i t \Box}$
  to denote
  the hyperbolic Schr\"odinger propagator.

\subsection{Hyperbolic counting estimate}
\label{SUB:counting}

Our main goal
in this subsection
 is to establish the following
sharp 
 counting estimate associated with the cubic HNLS  \eqref{HNLS1}.

\begin{lemma}[hyperbolic counting estimate]\label{LEM:count1}
Given $n \in \Z^2$ and $\mu \in \Z$, let 
$\G_\mu^-(n)$ be as in~\eqref{GG0}.
Then, given any $\ta > 0$, there exists $C > 0$, 
independent of  $n \in \Z^2$ and $\mu \in \Z$,  
such that 
    \begin{align}
\big|  \G_\mu^-(n) \cap \{ |n_1| \sim N_1, \, |n_3|\sim N_3\}  \big|
&  \leq CN_1N_3\max(N_1, N_3)^{\ta},   
\label{count1} \\ 
\big|  \G_\mu^-(n) \cap \{ |n_1| \sim N_1, \, |n_2|\sim N_2\}  \big|
&  \leq CN_1N_2\max(N_1, N_2)^{\ta}
\label{count2}
    \end{align}

\noi
for any  dyadic $N_1, N_2, N_3 \ge 1$.
\end{lemma}

Lemma \ref{LEM:count1}
plays a crucial role
in establishing 
local well-posedness of \eqref{HNLS1}
via the Fourier restriction norm method (Theorem \ref{THM:1}\,(i))
and
unconditional 
local well-posedness
of the normal form equation \eqref{NF1a}
 (Theorem \ref{THM:3}\,(i)).
Before proceeding to a proof of Lemma \ref{LEM:count1}, 
let us recall the 
 divisor counting bound; see \cite[Lemma 4.4]{DNY2} for a proof.

\begin{lemma}\label{LEM:div1}

Given $m\in\Z\setminus \{0\}$
 $a_0,b_0\in \Z$, and $M, N > 0$
we have
\begin{equation*}
\big|\big\{(a, b)\in \Z^2: m=ab,\,\,|a-a_0|\leq M,\,\,|b-b_0|\leq N\big\}\big|
= O(M^\ta N^\ta)
\end{equation*}

\noi
for any $\ta > 0$.
Here, the implicit constant is independent of $m$, $a_0$,  $b_0$, 
$M$, and $N$.

\end{lemma}

We now present a proof of the hyperbolic counting estimate (Lemma \ref{LEM:count1}).

\begin{proof}[Proof of Lemma \ref{LEM:count1}]

We first prove \eqref{count1}. 
Write $n = (j, k)$ and $n_\l = (j_\l, k_\l)$, $\l = 1, 2, 3$.
From \eqref{GG0} with \eqref{GG0a}, 
we have 
 \begin{align}
2(j - j_1)(j-j_3) -  2(k - k_1) (k-k_3)  = \mu
\label{count3}
 \end{align}
 
\noi
on $\G_\mu^-(n)$. 
Fix $k_1, k_3 \in \Z$ with $|n_1|\sim N_1$ and $|n_3|\sim N_3$, 
which has $O(N_1 N_3)$ many choices.
Note that, from \eqref{count3}, 
the value of 
$m = (j - j_1)(j-j_3) = \frac \mu 2 + (k-k_1)(k-k_3)$ is determined, 
since $k, \mu, k_1, $ and $k_3$ are fixed.
Suppose that  $m = (j - j_1)(j-j_3)  \ne 0$.
Then, it follows from Lemma \ref{LEM:div1} that,  
 given $m = (j - j_1)(j-j_3) \in \Z\setminus\{0\}$,    
there are $O(\max(N_1,  N_3)^\ta)$ many choices
for $j_1$ and $j_3$, 
which determines $n_1$ and $n_3$ (and hence $n_2$
under the condition~\eqref{GG0b}).
This yields the bound \eqref{count1}
 if $(j - j_1)(j-j_3)  \ne 0$.
When $(k - k_1) (k-k_3)\ne 0$,  
 a similar argument yields the bound \eqref{count1}.
 
 Hence, it remains to consider the case: 
 \begin{align*}
(j - j_1)(j-j_3) =  (k - k_1) (k-k_3)  = 0.
\end{align*} 

\noi
Since $n\neq n_1, n_3$, it suffices to consider  the following cases: 
\begin{align*}
\text{(a)} \  \ j - j_1 = 0 \text{ and }k - k_3 = 0
\qquad \text{and}
\qquad 
\text{(b)}\ \ 
j - j_3 = 0 \text{ and }k - k_1 = 0.
\end{align*}

\noi
By symmetry, we only consider Case (a).
In this case, given $n = (j, k) \in \Z^2$, $j_1$ and $k_3$ are determined.
Under the condition $|n_1|\sim N_1$ and $|n_3|\sim N_3$, 
this leaves $O(N_1N_3)$ many choices
for $k_1 $ and $j_3$, 
which yields \eqref{count1}.

\medskip
Next, we prove \eqref{count2}.
From \eqref{GG0} with \eqref{GG0a}, 
we have 
 \begin{align}
2(j - j_1)(j_1-j_2) -  2(k - k_1) (k_1-k_2)  = \mu
\label{count5}
 \end{align}

\noi
on $\G_\mu^-(n)$. 
Fix $k_1, k_2 \in \Z$ with $|n_1|\sim N_1$ and $|n_2|\sim N_2$, 
which has $O(N_1 N_2)$ many choices.
Note that, from \eqref{count5}, 
the value of 
$m = (j - j_1)(j_1-j_2) = \frac \mu 2 + (k-k_1)(k_1-k_2)$ is determined, 
since $k, \mu, k_1, $ and $k_2$ are fixed.
Suppose that  $m = (j - j_1)(j_1-j_2)  \ne 0$.
Then, it follows from Lemma \ref{LEM:div1} that,  
 given $m = (j - j_1)(j_1-j_2) \in \Z\setminus\{0\}$,    
there are $O(\max(N_1,  N_2)^\ta)$ many choices
for $j_1$ and $j_2$, 
which determines $n_1$ and $n_2$ (and hence $n_3$
under the condition~\eqref{GG0b}).
This yields the bound \eqref{count2}
 if $(j - j_1)(j_1-j_2)  \ne 0$.
When $(k - k_1) (k_1-k_2)\ne 0$,  
 a similar argument yields the bound \eqref{count2}.

 Hence, it remains to consider the case: 
 \begin{align*}
(j - j_1)(j_1-j_2) =  (k - k_1) (k_1-k_2)  = 0.
\end{align*}

\noi
Since $n\neq n_1, n_3$, it suffices to consider 
 the following cases: 
\begin{align*}
\text{(c)} \  \ j - j_1 = 0 \text{ and }k_1 - k_2 = 0
\qquad \text{and}
\qquad 
\text{(d)}\ \ 
j_1 - j_2 = 0 \text{ and }k - k_1 = 0.
\end{align*}

\noi
By symmetry, we only consider Case (c).
In this case, given $n = (j, k) \in \Z^2$, $j_1$ is  determined.
Under the condition $|n_1|\sim N_1$, 
this leaves $O(N_1)$ many choices
for $k_1 $, which determines $k_2$.
Then, under the condition  $|n_2|\sim N_2$, 
this leaves $O(N_2)$ many choices
for $j_2 $, 
which yields~\eqref{count2}.
\end{proof}

\begin{remark}\label{REM:sharp}\rm

We point out that the proof of Lemma \ref{LEM:count1}
also applies to the elliptic case
(namely, $\G_\mu^+(n)$ in \eqref{GG0} also satisfies the bounds \eqref{count1} and \eqref{count2}).
The only change appears
in \eqref{count3} and \eqref{count5}, 
where the $-$ sign becomes the $+$ sign, 
but this does not affect the argument.
See also \cite{BO96, DNY2,Kishi21, OW3}
for further counting estimates in the elliptic case.

\end{remark}

\section{Normal form reductions}
\label{SEC:NF1}

In this section, we  go over 
an infinite iteration of normal form reductions
and formally derive the normal form equation \eqref{NF1a}
with explicit definitions of the multilinear operators
$\Nf_0^{(j)}$, $\Nf_1^{(j)}$, and $\Rf^{(j)}$; see \eqref{NF1b}.
Since the formal procedure remains the same as
those presented in \cite{GKO, OW2}, 
we keep our presentation brief.
In the next section, 
we estimate
the relevant multilinear operators.

\subsection{First step}
\label{SUBSEC:NF1a}

Define the trilinear operators $\NN ^{(1)}$ and $\RR ^{(1)}$ as follows
\begin{equation}
\begin{split}
\NN ^{(1)}(\uu  _1, \uu  _2, \uu  _3)& =i\sum_{n\in\Z^2}e^{in \cdot x}
   \sum_{\substack{n=n_1-n_2+n_3\\n_2\neq n_1, n_3}}e^{it\Phi (\bar n) }\widehat{\uu  }_1(n_1)\overline{\widehat{\uu  }_2(n_2)}\widehat{\uu  }_3(n_3), \\
     \RR ^{(1)}(\uu  _1, \uu  _2, \uu  _3) &
 = i \, \Rf_1(\uu_1, \uu_2, \uu_3)
 + i \, \Rf_2(\uu_1, \uu_2, \uu_3), 
\end{split}
\label{non3} 
\end{equation}

\noi
where $\Rf_1$ and $ \Rf_2$ are  as in \eqref{non1}.
Then, from 
\eqref{HNLS2}
and 
\eqref{non3}, we have 
\begin{align*}
  \NN ^{(1)}(\uu  )=\NN ^{(1)}(\uu  , \uu  , \uu  )
  \qquad \text{and}
  \qquad \RR ^{(1)}(\uu  )=\RR ^{(1)}(\uu  , \uu  , \uu  ), 
\end{align*}

\noi
which is consistent with 
the short-hand notation \eqref{short1}.

\medskip

\noi
{\bf Conventions:}
(i) For simplicity of notation, 
given a multilinear operator $T$, 
we use $T(f)(n)$ to denote the Fourier coefficients of $T(f)$
in the following.
We also use $\ft \uu_n = \ft \uu_n(t)$  
to denote $\ft \uu (t, n)$.
  
\smallskip

\noi
(ii)
Multilinear operators that appear in this section
are non-autonomous, i.e.~they depend on a parameter $t \in \R$.
They, however, satisfy estimates uniformly in time
and hence we simply suppress their time dependence
(as we already did for $\NN^{(1)}$ in \eqref{non3}).
In specifying $t$-dependence, 
we write $S(f)(t)$ and $S(f)(t, n)$, etc.

\medskip

With these conventions, we can write \eqref{HNLS2} as 
\begin{align}
 \dt \ft \uu_n  =\NN ^{(1)}(\uu  )(n)+\RR ^{(1)}(\uu  ) (n). 
\label{HNLS3}
\end{align}

\noi
The following lemma on 
the trivial resonant part $\RR ^{(1)}$ 
follows from the embeddings:
$\l_n^p  \subset  \l_n^{3p}$, 
and $\F L^{s, p}(\T^2) \subset L^2(\T^2)$.

\begin{lemma}\label{LEM:R1} 
Let $1\leq p\leq \infty$ and $s \ge 0$ with $s > 1 - \frac 2p$.
Then, we have
\begin{align}
    \| \RR ^{(1)}(\uu_1, \uu_2, \uu_3)\|_{\mathcal{F}L^{s,p}}
\les \prod_{j=1}^3 \| \uu_j\|_{\mathcal{F}L^{s,p}}.
\label{R1}
\end{align}
\end{lemma}

\begin{remark}\label{REM:con1}\rm
In the following, we establish various multilinear estimates.
For simplicity of  notations, 
we only state and prove estimates when all arguments agree
with the understanding that they can be easily extended 
to multilinear estimates; see Remark \ref{REM:diff2}.
Under this convention,
we can write \eqref{R1} as 
\begin{align*}
\| \RR^{(1)}(\uu)\|_{\F L^{s, p} }
\les  \|\uu\|_{\F L^{s, p}}^3.
\end{align*}

\end{remark}

As in \cite{GKO, OW2}, 
we now divide the essential nonlinear part 
 $\NN ^{(1)}$
 into the nearly resonant part and the (highly) non-resonant part.
Fix $1 <  p<\infty$ and $K\geq 1$ (to be chosen appropriately, depending on $\|u(0)\|_{\FL^{s, p}}$;
  see \cite[Proof of Theorem 1.9]{OW2}).
Then, given $n \in \Z^2$, 
we set 
\begin{align}
 A_1(n)=\bigcup_{|\mu|\le (3K)^{4p}}\G_\mu^-(n)
 \qquad \text{and} \qquad A_1 = \bigcup_{n \in \Z^2} A_1(n), 
\label{A1}
\end{align}

\noi
where 
$\G_\mu^-(n)$ is as in \eqref{GG0}.
Then, we define the nearly resonant part 
 $\NN ^{(1)}_1$ of
 $\NN ^{(1)}$
(and the non-resonant part $ \NN ^{(1)}_2$)  
 as  its restriction  onto $A_1$
(and its restriction onto $A_1^c$)
such that 
\begin{equation}
  \NN ^{(1)}=\NN ^{(1)}_1+\NN ^{(1)}_2 . 
\label{NN1}     
\end{equation}

\noi
Thanks to the modulation cutoff and the hyperbolic counting estimate
(Lemma \ref{LEM:count1}), 
the nearly resonant part $\NN^{(1)}_1$ satisfies
a good bound, 
which  is a special case of Lemma \ref{LEM:N1j};
see Subsection \ref{SUBSEC:NF2b} for a proof.


\begin{lemma}\label{LEM:N11} 
Let $1\leq p<\infty$ and  $s>1-\frac{1}{p}$.
Then,  we have
\begin{align*}
\| \NN _1^{(1)}(f_1, f_2, f_3) (t)\|_{\mathcal{F}L^{s,p}}
\lesssim K^{4(p-1)}\prod_{j = 1}^3 \|f_j\|^3_{\mathcal{F} L^{s,p}}, 
\end{align*}

\noi
uniformly in $t \in \R$.

\end{lemma}

We now apply a normal for reduction to the non-resonant part  $\NN^{(1)}_2$
in the form of ``differentiation by parts'' \cite{BIT}:
\begin{equation}
\begin{aligned}
\NN ^{(1)}_2(\uu  )(t, n)
&=\sum_{A_1(n)^c}\partial_t\bigg(\frac{e^{it\Phi(\bar n)}}{\Phi(\bar n)}\bigg)\widehat{\uu  }_{n_1}\overline{\widehat{\uu  }}_{n_2} \widehat{\uu  }_{n_3} \\
&=\sum_{A_1(n)^c}\partial_t\Big(\frac{e^{it\Phi (\bar n)}}{\Phi (\bar n)}\widehat{\uu  }_{n_1}\overline{\widehat{\uu  }}_{n_2} \widehat{\uu  }_{n_3} \Big)- \sum_{A_1(n)^c}\frac{e^{it\Phi (\bar n)}}{\Phi (\bar n)}\partial_t(\widehat{\uu  }_{n_1}\overline{\widehat{\uu  }}_{n_2} \widehat{\uu  }_{n_3})\\&=\partial_t\Big(\sum_{A_1(n)^c}\frac{e^{it\Phi (\bar n)}}{\Phi (\bar n)}\widehat{\uu  }_{n_1}\overline{\widehat{\uu  }}_{n_2} \widehat{\uu  }_{n_3} \Big)
- \sum_{j = 1}^3
\sum_{A_1(n)^c}\frac{e^{it\Phi (\bar n)}}{\Phi (\bar n)}
\dt \ft \uu_{n_ j}^* 
\prod_{\substack{k = 1\\k \ne j }}^3 
\ft \uu_{n_k}^* \\
&=: \partial_t\NN _0^{(2)}(\uu  )(t, n)+\wt \NN^{(2)}(\uu  )(t, n), 
\end{aligned}
\label{NN21}      
 \end{equation}

\noi
 where 
 \begin{align*}
 \ft \uu_{n_j}^*
= \begin{cases}
 \ft \uu_{n_j},  & \text{when $j = 1, 3$,}\rule[-3mm]{0pt}{0pt} \\
\cj{ \ft \uu_{n_j}}, & \text{when $j = 2$.}
\end{cases}
\end{align*}

 \noi
By  substituting \eqref{HNLS3},   
we have 
 \begin{equation}\label{NN21a}
\begin{aligned}
\widetilde{\NN }^{(2)}(\uu  )(t, n)
& =  \NN ^{(2)}(\uu  )(t, n)+ \RR ^{(2)}(\uu  )(t, n)\\
: \! &=
- \sum_{j = 1}^3
\sum_{A_1(n)^c}\frac{e^{it\Phi (\bar n)}}{\Phi (\bar n)}
\big( \NN ^{(1)}(\uu  )\big)^*(n_j)
\prod_{\substack{k = 1\\k \ne j }}^3 
\ft \uu_{n_k}^* \\
& \quad  
- \sum_{j = 1}^3
\sum_{A_1(n)^c}\frac{e^{it\Phi (\bar n)}}{\Phi (\bar n)}
\big( \RR ^{(1)}(\uu  )\big)^*(n_j)
\prod_{\substack{k = 1\\k \ne j }}^3 
\ft \uu_{n_k}^*.
%
%
     \end{aligned}
 \end{equation}

\noi
See
Lemmas \ref{LEM:N0j} and \ref{LEM:Rj}
for estimates on  the boundary term $\NN_0^{(2)}(\uu)$ in \eqref{NN21}
and $\RR ^{(2)}(\uu  )$ in~\eqref{NN21a}.
As for the term  $\NN ^{(2)}(\uu)$ in \eqref{NN21a}, 
we need to split it into the nearly resonant and non-resonant parts
and apply a normal form reduction on the non-resonant part.
We then iterate this process indefinitely.

 \begin{remark} \label{REM:just1} \rm 
At the third step of \eqref{NN21}, 
we formally switched 
the order of the sum and the time differentiation
in the first term and also applied the product rule
on the summand in the second term.
As observed in  \cite{GKO}, 
these can be justified in the classical sense 
if $\uu \in  C([0, T]; L^3(\T))$;
see also the embeddings 
\eqref{emb0} and 
\eqref{emb1}.
A similar comment applies to the computation in \eqref{NN2j}.

 \end{remark}

 \subsection{General step}
 \label{SUBSEC:NF1b}

Before proceeding further, we first 
recall the notion of (ternary) ordered trees and relevant  definitions
from \cite{GKO}; see also \cite{OW2}.

\begin{definition} \label{DEF:tree1} \rm
\textup{(i)} 
Given a partially ordered set $\TT$ with partial order $\leq$, 
we say that $b \in \TT$ 
with $b \leq a$ and $b \ne a$
is a child of $a \in \TT$,
if  $b\leq c \leq a$ implies
either $c = a$ or $c = b$.
If the latter condition holds, we also say that $a$ is the parent of $b$.

\smallskip
\noi
\textup{(ii)}
A tree $\TT$ is a finite partially ordered set satisfying
the following properties.
\begin{itemize}
\item Let $a_1, a_2, a_3, a_4 \in \TT$.
If $a_4 \leq a_2 \leq a_1$ and  
$a_4 \leq a_3 \leq a_1$, then we have $a_2\leq a_3$ or $a_3 \leq a_2$,

\item
A node $a\in \TT$ is called terminal, if it has no child.
A non-terminal node $a\in \TT$ is a node 
with  exactly three children denoted by $a_1, a_2$, and $a_3$,

\item There exists a maximal element $r \in \TT$ (called the root node) such that $a \leq r$ for all $a \in \TT$.
We assume that the root node is non-terminal,

\item $\TT$ consists of the disjoint union of $\TT^0$ and $\TT^\infty$,
where $\TT^0$ and $\TT^\infty$
denote  the collections of non-terminal nodes and terminal nodes, respectively.
\end{itemize}

\noi
The number $|\TT|$ of nodes in a tree $\TT$ is $3j+1$ for some $j \in \NB$,
where $|\TT^0| = j$ and $|\TT^\infty| = 2j + 1$.
Let us denote  the collection of trees in the $j$th generation
by $T(j)$:
\begin{equation*}
T(j) := \{ \TT : \TT \text{ is a tree with } |\TT| = 3j+1 \}.
\end{equation*}

\noi
Note that $\TT \in T(j)$ contains  $j$ parental nodes.

\smallskip

\noi
(iii) (ordered tree) We say that a sequence $\{ \TT_j\}_{j = 1}^J$ is a chronicle of $J$ generations, 
if 
\begin{itemize}
\item $\TT_j \in {T}(j)$ for each $j = 1, \dots, J$,
\item  $\TT_{j+1}$ is obtained by changing one of the terminal
nodes in $\TT_j$ into a non-terminal node (with three children), $j = 1, \dots, J - 1$.
\end{itemize}

\noi
Given a chronicle $\{ \TT_j\}_{j = 1}^J$ of $J$ generations,  
we refer to $\TT_J$ as an {\it ordered tree} of the $J$th generation.
We denote the collection of the ordered trees of the $J$th generation
by $\mathfrak{T}(J)$.
Note that the cardinality of $\mathfrak{T}(J)$ is given by 
\begin{equation} \label{cj1}
 |\mathfrak{T}(J)| = 1\cdot3 \cdot 5 \cdot \cdots \cdot (2J-1)
 = (2J-1)!!.
 \end{equation}

\smallskip

\noi
(iv) Let $\TT_J$ be an ordered tree of the $J$th generation
with a chronicle $\{ \TT_j\}_{j = 1}^J$.
Given $k = 1, \dots, J-1$, 
we define the projection $\pi_k: 
\mathfrak{T}(J) \to \mathfrak{T}(k)$ by setting
\begin{align}
\pi_k(\TT_J) = \TT_k
\quad \text{endowed with the chronicle $\{ \TT_j\}_{j = 1}^k$}.
\label{cj2}
\end{align}

\end{definition}

Roughly speaking, 
 ordered trees are the trees
 that remember how they ``grew''.
This property turned out to be convenient in encoding 
successive applications
of the product rule for differentiation \cite{GKO, OW2}.
In the following, we simply refer to an ordered tree $\TT_J$ of the $J$th generation
but 
it is understood that there is an underlying chronicle $\{ \TT_j\}_{j = 1}^J$.

\begin{remark}\rm

We point out that  ordered trees defined above 
 are very similar to the heap-ordered trees in 
 \cite{FU}; they are just a different way to describe the arborification.
 See also Remark \ref{REM:Bruned}.

\end{remark}

Given a tree $\TT$, we associate each terminal node 
$a\in \TT^\infty$ with the Fourier coefficient (or its complex conjugate) of the interaction representation $\uu$ and sum over all possible frequency assignments. 
In order to do this, we introduce the index function $\n$ assigning frequencies 
to all the nodes in $\TT$ in a consistent manner.

\begin{definition}[index function] \label{DEF:tree4} \rm
Given an ordered tree $\TT$ (of the $J$th generation for some $J \in \NB$), 
we define an index function ${\bf n}: \TT \to \Z^2$ such that,
\begin{itemize}
\item[(i)] $n_a = n_{a_1} - n_{a_2} + n_{a_3}$ for $a \in \TT^0$,
where $a_1, a_2$, and $a_3$ denote the children of $a$,
\item[(ii)] $\{n_a, n_{a_2}\} \cap \{n_{a_1}, n_{a_3}\} = \emptyset$ for $a \in \TT^0$,

\item[(iii)] 
$|\mu_1|:=\big||n_r|_-^2-|n_{r_1}|_-^2+|n_{r_2}|_-^2-|n_{r_3}|_-^2\big|>(3K)^{4p}$,\footnote{Recall that we are on $A_1(n)^c$. See \eqref{A1}.} where $r$ is the root node with its children $r_1$, $r_2$, and $r_3$.

\end{itemize}

\noi
where  we identified ${\bf n}: \TT \to \Z$ 
with $\{n_a \}_{a\in \TT} \in (\Z^2)^\TT$. 
We use 
$\mathfrak{N}(\TT) \subset (\Z^2)^\TT$ to denote the collection of such index functions ${\bf n}$.

\end{definition}

\medskip

Given an ordered tree 
$\TT_J$ of the $J$th generation with the chronicle $\{ \TT_j\}_{j = 1}^J$ 
and associated index functions ${\bf n} \in \mathfrak{N}(\TT_J)$,
 we use superscripts to 
  denote  ``generations'' of frequencies.

Fix ${\bf n} \in \mathfrak{N}(\TT_J)$.
Consider $\TT_1$ of the first generation.
Its nodes consist of the root node $r$
and its children $r_1, r_2, $ and $r_3$. 
We define the first generation of frequencies by
\begin{align*}
\big(n^{(1)}, n^{(1)}_1, n^{(1)}_2, n^{(1)}_3\big) :=(n_r, n_{r_1}, n_{r_2}, n_{r_3}).
\end{align*}

\noi
 The ordered tree $\TT_2$ of the second generation is obtained from $\TT_1$ by
changing one of its terminal nodes $a = r_k \in \TT^\infty_1$ for some $k \in \{1, 2, 3\}$
into a non-terminal node.
Then, we define
the second generation of frequencies by
\[\big(n^{(2)}, n^{(2)}_1, n^{(2)}_2, n^{(2)}_3\big) :=(n_a, n_{a_1}, n_{a_2}, n_{a_3}).\]

\noi
Note that  we have $n^{(2)} = n_k^{(1)} = n_{r_k}$ for some $k \in \{1, 2, 3\}$.
This corresponds to introducing a new set of frequencies
after the first differentiation by parts.

After  $j - 1$ steps, the ordered tree $\TT_j$ 
of the $j$th generation is obtained from $\TT_{j-1}$ by
changing one of its terminal nodes $a  \in \TT^\infty_{j-1}$
into a non-terminal node.
Then, we define
the $j$th generation of frequencies by
\[\big(n^{(j)}, n^{(j)}_1, n^{(j)}_2, n^{(j)}_3\big) :=(n_a, n_{a_1}, n_{a_2}, n_{a_3}).\]

\noi
Note that these frequencies
satisfies (i) and (ii) 
in Definition \ref{DEF:tree4}.

Lastly, we use $\mu_j$  to denote the corresponding modulation function introduced at the $j$th generation:
\begin{align}\label{mu1}
    \mu_j=\mu_j(n^{(j)}, n^{(j)}_1, n^{(j)}_2, n^{(j)}_3):= |n^{(j)}|_-^2-|n^{(j)}_1|_-^2+|n^{(j)}_2|_-^2-|n^{(j)}_3|_-^2.
\end{align}

\noi
Then,  by setting
\begin{align}\label{mu2}
\wt \mu_j =\sum_{k=1}^j\mu_k, 
\end{align}

\noi
we  define the nearly resonant set $A_j$ for the $j$th generation:
\begin{align}
 A_j=\big\{\mathbf{n}\in \mathfrak{N}(\TT_j): |\widetilde{\mu}_j |\leq ((2j+1)K)^{4p}\big\} .  
\label{Aj}
\end{align}

\noi
Compare this with \eqref{A1}.

\begin{remark} \label{REM:terminal}
\rm 

Note that ${\bf n} = \{n_a\}_{a\in\TT}$ is completely determined
once we specify the values $n_a$ for $a \in \TT^\infty$.
Moreover, 
given an ordered tree $\TT_j$ of the $j$th generation, we have 
\begin{align*}
n_ r = \sum_{a \in \TT_j^\infty} \kk_a n_{a}
\end{align*}

\noi
for some appropriate choices of $\kk_a \in \{+1, -1\}$.
Recalling 
$|\TT^\infty_j| = 2j + 1$, 
this implies that the symbol 
\[\sum_{\substack{\mathbf{n}\in \mathfrak{N}(\TT_{j})\\\mathbf{n}_r=n}}\]

\noi
means that we are performing $2j$ summations
on $n_a$, $a \in \TT_j^\infty \setminus \{b\}$
by first fixing $b\in \TT_j^\infty$.

\end{remark}

\begin{remark}\rm
For simplicity of notation, 
we may drop the minus signs, the complex number $i$, 
and the complex conjugate sign in the following
when they do not play an important role.
\end{remark}

\medskip

Let us now go back to the normal form procedure.
As mentioned above, we split $\NN^{(2)}$ in~\eqref{NN21a}
into the nearly
resonant part
 $\NN^{(2)}_1$ (as the restriction
 onto $A_2$ defined in \eqref{Aj}) 
and  the non-resonant part
 $\NN^{(2)}_2$ and apply the second normal form reduction 
 to 
$\NN^{(2)}_2$.
This formal procedure is precisely the same as those in \cite{GKO, OW2}.

After $(j-1)$ steps,  we are left with the following the non-resonant part:
\begin{align}\label{N20}
 \NN _2^{(j)}(\uu  )(n)=\sum_{\TT_{j}\in \mathfrak{T}(j)} \sum_{\substack{\mathbf{n}\in \mathfrak{N}(\TT_{j})\\\mathbf{n}_r=n}} \ind_{\bigcap_{k=1}^{j}\text{A}_k^c}\frac{e^{i t\widetilde{\mu}_{j}}}{\prod_{k=1}^{j-1}\widetilde{\mu}_k} \prod_{a\in \TT^{\infty}_{j}}\widehat{\uu  }_{n_a},
\end{align}

\noi
to which we apply a normal form reduction:
\begin{equation}\label{NN2j}
 \begin{aligned}
     \NN _2^{(j)}(\uu  )(n)&=\partial_t\Big(\sum_{\TT_j\in\mathfrak{T}(j)}\sum_{\substack{\mathbf{n}\in\mathfrak{N}(\TT_j)\\ \mathbf{n}_r=n}} \ind_{\bigcap_{k=1}^{j}\text{A}_k^c}\frac{e^{i t\widetilde{\mu}_j}}{\prod_{k=1}^j\widetilde{\mu}_k} \prod_{a\in \TT^{\infty}_{j}}\widehat{\uu  }_{n_a}\Big)\\&\quad +\sum_{\TT_j\in\mathfrak{T}(j)}\sum_{\substack{\mathbf{n}\in\mathfrak{N}(\TT_j)\\\mathbf{n}_r=n}}\sum_{b\in\TT_j^{\infty}}\ind_{\bigcap_{k=1}^{j}\text{A}_k^c}\frac{e^{it\widetilde{\mu}_j}}
     {\prod_{k=1}^j\widetilde{\mu}_k} \RR ^{(1)}(\uu  )(n_b)\prod_{a\in \TT^{\infty}_{j}\setminus \{b\}}\widehat{\uu  }_{n_a}\\&\quad +\sum_{\TT_j\in\mathfrak{T}(j)}\sum_{\substack{\mathbf{n}\in\mathfrak{N}(\TT_j)\\\mathbf{n}_r=n}}\sum_{b\in\TT_j^{\infty}}\ind_{\bigcap_{k=1}^{j}\text{A}_k^c}\frac{e^{it\widetilde{\mu}_j}}{\prod_{k=1}^j\widetilde{\mu}_k} \NN ^{(1)}(\uu  )(n_b)\prod_{a\in \TT^{\infty}_{j}\setminus \{b\}}\widehat{\uu  }_{n_a}\\&=\partial_t\Big(\sum_{\TT_j\in\mathfrak{T}(j)}\sum_{\substack{\mathbf{n}\in\mathfrak{N}(\TT_j)\\\mathbf{n}_r=n}}\ind_{\bigcap_{k=1}^{j}\text{A}_k^c}\frac{e^{it\widetilde{\mu}_j}}{\prod_{k=1}^j\widetilde{\mu}_k} \prod_{a\in \TT^{\infty}_{j}}\widehat{\uu  }_{n_a}\Big)\\& \quad+\sum_{\TT_j\in\mathfrak{T}(j)}\sum_{\substack{\mathbf{n}\in\mathfrak{N}(\TT_j)\\\mathbf{n}_r=n}}\sum_{b\in\TT_j^{\infty}}\ind_{\bigcap_{k=1}^{j}\text{A}_k^c}\frac{e^{it\widetilde{\mu}_j}}{\prod_{k=1}^j\widetilde{\mu}_k} \RR ^{(1)}(\uu  )(n_b)\prod_{a\in \TT^{\infty}_{j}\setminus \{b\}}\widehat{\uu  }_{n_a}\\&\quad+\sum_{\TT_{j+1}\in\mathfrak{T}(j+1)}\sum_{\substack{\mathbf{n}\in\mathfrak{N}(\TT_{j+1})\\\mathbf{n}_r=n}}\ind_{\bigcap_{k=1}^{j}\text{A}_k^c}\frac{e^{i t \widetilde{\mu}_{j+1}}}{\prod_{k=1}^j\widetilde{\mu}_k} \prod_{a\in \TT^{\infty}_{j+1}}\widehat{\uu  }_{n_a}\\&=: \partial_t\NN _0^{(j+1)}(\uu  )(n)+\RR ^{(j+1)}(\uu  )(n)+\NN ^{(j+1)}(\uu  )(n).
 \end{aligned}   
\end{equation}

\noi
We then  split $\NN ^{(j+1)}$ as 
\begin{align}\label{NN2ja}
\NN ^{(j+1)}=\NN ^{(j+1)}_1+\NN ^{(j+1)}_2,  
\end{align}
where $\NN ^{(j+1)}_1$ is the restriction of $\NN ^{(j+1)}$ onto $A_{j+1}$ and 
$\NN ^{(j+1)}_2:=\NN ^{(j+1)}-\NN ^{(j+1)}_1$, 
and apply a normal form reduction to the non-resonant part
$\NN ^{(j+1)}_2$.

\subsection{Normal form equation}
\label{SUBSEC:NF1c}

After applying $J-1$ normal form reductions, 
we arrive at the following formulation:
\begin{align}\label{NFE1}
\begin{split}
 \uu (t)     
&  =\uu  (0)+\sum_{j=2}^J\NN _0^{(j)}(\uu  )(t)- \sum_{j=2}^J\NN _0^{(j)}(\uu  )(0) 
 \\
 &\quad +\int_0^t
 \bigg\{ \sum_{j=1}^J\NN _1^{(j)}(\uu  )(t')
 +\sum_{j=1}^J\RR ^{(j)}(\uu  )(t')\bigg\}dt'
+\int_0^t \NN _2^{(J)}(\uu  )(t') dt'.
\end{split}
\end{align}

The following lemma shows that the error term
(the last term on the right-hand side of~\eqref{NFE1})
tends to $0$ as $J \to \infty$.
See Subsection \ref{SUBSEC:NF2c} for a proof.

\begin{lemma}\label{LEM:N2j}
 Let $\NN _2^{(J)}$ be as in \eqref{N20}.
 Suppose that 
  $ 1< p < \infty$ and $s \in \R$  satisfy 
 \eqref{emb0a} or \eqref{emb2}.
Then,  
 we have
\begin{align}\label{N2jx}
\| \NN _2^{(J)}(\uu  )(t)\|_{\mathcal{F}L^{\infty}(\T^2)}\longrightarrow 0 , 
\end{align}

\noi
as $J\rightarrow\infty$.
The convergence is 
uniform in $t \in \R$
and
locally uniform in 
$\mathcal{F}L^{s,p}(\T^2)$.

\end{lemma}

Hence, by taking $J \to \infty$ in \eqref{NFE1}, 
we obtain the normal form equation at the level of the interaction representation
$\uu(t) = S(-t) u(t)$:
\begin{equation}
\begin{aligned}
 \uu  (t)& =\uu  (0)+\sum_{j=2}^{\infty}\NN _0^{(j)}(\uu  )(t)- \sum_{j=2}^{\infty}\NN _0^{(j)}(\uu  )(0) 
 \\&\quad +\int_0^t
 \bigg\{ \sum_{j=1}^{\infty}\NN _1^{(j)}(\uu  )(t')
 +\sum_{j=1}^{\infty}\RR ^{(j)}(\uu  )(t')\bigg\}dt'
    \end{aligned}
\label{NF1}
\end{equation}

\noi
where  $\{\NN _0^{(k)}\}_{k=2}^{\infty}$, $\{\NN _1^{(k)}\}_{k=1}^{\infty}$, and $\{\RR ^{(k)}\}_{k=1}^{\infty}$ are as in \eqref{NN2ja}.
Finally, by setting
\begin{align}
\begin{split}
\Nf^{(j)}_0(u)(t) & =    S(t) \NN^{(j)}_0(S(-\,\cdot)u)(t),\\
\Nf^{(j)}_1(u)(t) & =  S(t) \NN^{(j)}_1(S(-\,\cdot)u)(t),\\ 
\Rf^{(j)}(u)(t) & =  S(t)  \RR^{(j)}(S(-\,\cdot)u)(t), 
\end{split}
\label{NF1b}
\end{align}

\noi
we obtain the normal form equation \eqref{NF1a} from \eqref{NF1}.
As observed  in \cite{OW2}, 
 the multilinear operators 
$ \Nf^{(j)}_0$, $\Nf^{(j)}_1$, and
$\Rf^{(j)}$
are autonomous.

\begin{remark}\label{REM:UU}\rm

As in \cite{GKO, OW2}, 
the convergence \eqref{N2jx} in Lemma \ref{LEM:N2j}
takes place in a weaker topology (i.e.~$\FL^{\infty}(\T^2)$), 
while we assume a higher regularity for the input
function $\uu \in \FL^{s, p}(\T^3) \subset L^3(\T^2) \cup \F L^\frac32 (\T^2)$.
We, however, point out that we also need $s > 1- \frac 1p$
to bound the multilinear terms in \eqref{NF1}, 
and hence the regularity \eqref{emb3} is required.


\end{remark}

\begin{remark}\label{REM:Bruned} \rm
In a recent preprint
\cite{Bruned}, 
Bruned
provided a derivation of the normal form equation 
(for the cubic elliptic NLS on $\T$) from the algebraic viewpoint.
This was done by 
combining the Fourier decorated trees introduced in 
\cite{BSch}, the Butcher series formalism, and the arborification coming from Ecalle's theory on moulds 
\cite{EV, FM}, 
showing that 
 Hopf algebras also appear naturally in the context of dispersive equations.
\end{remark}

\subsection{On the multilinear operators}
\label{SEC:NF1d}

We conclude this section by stating 
bounds on 
the multilinear 
operators appearing in our normal form reductions.
We prove these lemmas in the next section.

\begin{lemma}\label{LEM:N0j}
Let $\NN _0^{(j)}$ be as in \eqref{NN2j}.
 Let $1\leq p<\infty$ and  $s>1-\frac{1}{p}$. 
 Then, there exists constant $C_p>0$ such that   
 \begin{align}\label{N0x}
 \| \NN _0^{(j)}(\uu  )\|_{\mathcal{F}L^{s,p}(\T^2)}\leq 
 C_p\frac{K^{4(1-j)}}{((2j-1)!!)^{2}}\| \uu  \|^{2j-1}_{\mathcal{F}L^{s,p}(\T^2)}
\end{align}
for any integer $j\geq2$ and $K\geq 1$.
\end{lemma}

\begin{lemma}\label{LEM:Rj}
Let $\RR^{(j)}$ be as in \eqref{NN2j}.
Let $1\leq p<\infty$ and  $s>1-\frac{1}{p}$. Then,  there exists a constant $C_p>0$ such that   \begin{align}\label{Rjx}
 \| \RR ^{(j)}(\uu  )\|_{\mathcal{F}L^{s,p}(\T^2)}\leq C_p\frac{(2j-1)K^{4(1-j)}}{((2j-1)!!)^{2}}\| \uu  \|^{2j+1}_{\mathcal{F}L^{s,p}(\T^2)}
\end{align}

\noi
for any $j \in \NB$ and $K\geq 1$.
\end{lemma}

\begin{lemma}\label{LEM:N1j}
 Let $\NN _1^{(J)}$ be as in \eqref{NN2ja}.
Let $1\leq p<\infty$ and  $s>1-\frac{1}{p}$. Then, there exists a constant $C_p>0$ such that   
\begin{align}\label{Njx}
 \| \NN ^{(j)}_1(\uu  )(t)\|_{\mathcal{F}L^{s,p}(\T^2)}
 \leq C_p\frac{K^{4(p-j)}}{(2j-1)!!}\| \uu  \|^{2j+1}_{\mathcal{F}L^{s,p}} 
\end{align}

\noi
for any $j \in \NB$ and $K\geq 1$.
\end{lemma}

\section{Multilinear estimates}
\label{SEC:NF2}

In this section, we 
present proofs of Lemmas \ref{LEM:N2j}, 
\ref{LEM:N0j}, \ref{LEM:Rj}, and \ref{LEM:N1j}.
In Subsection \ref{SUBSEC:NF2a}, 
we
study the mapping property of 
the boundary term $\NN _0^{(j)}$  in \eqref{NN2j}
(Lemma \ref{LEM:N0j})
as a base case.
The bounds on the other operators
follow either as a direct corollary to 
Lemma \ref{LEM:N0j}
or as a slight modification of the proof
of Lemma \ref{LEM:N0j}.

\subsection{Boundary terms}
\label{SUBSEC:NF2a}
We first estimate the boundary term 
in \eqref{NN2j} (with $j + 1$ replaced by $j$):
\begin{align*}
 \NN _0^{(j)}(\uu  )(t, n)=\sum_{\TT_{j-1}\in \mathfrak{T}(j-1)} \sum_{\substack{\mathbf{n}\in \mathfrak{N}(\TT_{j-1})\\\mathbf{n}_r=n}} \ind_{\bigcap_{k=1}^{j-1}\text{A}_k^c}\frac{e^{i t \widetilde{\mu}_{j-1}}}{\prod_{k=1}^{j-1}\widetilde{\mu}_k} \prod_{a\in \TT^{\infty}_{j-1}}\widehat{\uu  }_{n_a}.
\end{align*}

\begin{proof}[Proof of Lemma \ref{LEM:N0j}]

By H\"older's inequality  with \eqref{cj1}, we obtain
\begin{equation}\label{N0j1}
  \begin{aligned}
\| \NN _0^{(j)}(\uu  )(t)\|_{\mathcal{F}L^{s,p}}
 &\leq \M_0^{(j)}
\sum_{\TT_{j-1}\in \mathfrak{T}(j-1)}
\bigg\| \bigg(\sum_{\substack{\mathbf{n}\in \mathfrak{N}(\TT_{j-1})\\\mathbf{n}_r=n}}\prod_{a\in \TT^{\infty}_{j-1}}\langle n_a \rangle^{sp}|\widehat{\uu  }_{n_a}|^p\bigg)^{\frac 1p} \bigg\|_{\ell^p_n}\\
& \leq (2j-3)!!
\cdot \M_0^{(j)}
\| \uu  \|^{2j-1}_{\mathcal{F}L^{s,p}}, 
\end{aligned}  
\end{equation}

\noi
uniformly in $t \in \R$, 
where $\M^{(j)}_0$ is defined by 
\begin{align}
\begin{split}
\M^{(j)}_0
& = 
\sup_{\substack{\TT_{j-1}\in \mathfrak{T}(j-1)\\n\in\Z^2}}
\M^{(j)}_0(\TT_{j-1}; n)\\
:\!& = \sup_{\substack{\TT_{j-1}\in \mathfrak{T}(j-1)\\n\in\Z^2}}
\bigg(\sum_{\substack{\mathbf{n}\in \mathfrak{N}(\TT_{j-1})\\ \mathbf{n}_r=n}} \frac{\langle n\rangle^{sp'}}{\prod_{a\in \TT^{\infty}_{j-1}}\langle n_a\rangle^{sp'}}\frac{\ind_{\bigcap_{k=1}^{j-1}\text{A}_k^c}}{\prod_{k=1}^{j-1}|\widetilde{\mu}_k|^{p'}}\bigg)^{\frac 1{p'}}.
\end{split}
\label{M0j}
\end{align}

\noi
Hence, it suffices to  prove
\begin{align}
\M^{(j)}_0 \le
C_p\frac{K^{4(1-j)}}{((2j-1)!!)^{3}}
\label{M0j2}
\end{align}

\noi
for some 
 constant 
 $C_p>0$ 
 depending only on $p$.

Let $\mu_k$ be as in \eqref{mu1}.
Suppose that 
we have 
\begin{align}
\mu_k = \al_k , 
\label{eta1}
\end{align}

\noi
for some  $\al_k \in \Z$, 
$k  = 1, \dots, j-1$. 
Then, from 
\eqref{mu2}, we have 
\begin{align*}
  |\widetilde{\mu}_k| = |\wt \al_k|
\end{align*}

\noi
for each $k = 1, \dots, j-1$, 
where $\wt \al_k$ is defined by 
\begin{align*}
\wt \al _k = \sum_{\l= 1}^k \al_\l.
\end{align*}

\noi
In view of \eqref{Aj}, 
under the assumption \eqref{eta1}, 
the frequencies appearing in \eqref{M0j}
satisfy 
\begin{align}
\prod_{k=1}^{j-1}|\widetilde{\mu}_k|
\ge   \prod_{k=1}^{j-1}
\max\big\{|\wt \al_k|, ((2k+1)K)^{4p}\big\}.    
\label{eta4}
\end{align}

\noi
Thus,  by H\"older's inequality with  \eqref{M0j}
and \eqref{eta4}, we have 
\begin{equation}\label{eta5}
\begin{aligned}
\big(\M^{(j)}_0(\TT_{j-1}; n)\big)^{p'}
&\le 
\S^{(j)}\cdot \K^{(j)}(\TT_{j-1}; n), 
\end{aligned}
\end{equation}

\noi
uniformly in 
$\TT_{j-1}\in \mathfrak{T}(j-1)$ and 
$n \in \Z^2$, 
where 
$\S^{(j)}$ and $\K^{(j)}(\TT_{j-1}; n)$ are  defined by 
\begin{align}
\begin{split}
\S^{(j)}
& = \sum_{\al_1,\dots,\al_{j-1}\in \Z}\frac{1}{\prod_{k=1}^{j-1}\max\big\{|\wt \al _k|, ((2k+1)K)^{4p}\big\}^{p'}}, \\
\K^{(j)}(\TT_{j-1}; n)
& = 
\sup_{\al_1,\dots,\al_{j-1}\in\Z}\sum_{\substack{\mathbf{n}\in \mathfrak{N}(\TT_{j-1})\\ \mathbf{n}_r=n\\ 
\mu_k = \al_k, k = 1, \dots, j - 1}}\frac{\langle n\rangle^{sp'}}{\prod_{a\in \TT^{\infty}_{j-1}}\langle n_a\rangle^{sp'}}.
\end{split}
\label{eta6}
\end{align}

Let $k = 1, \dots, j-1$.
Then, for fixed $\wt \al _{k - 1}$,  we have
\begin{align}
\begin{split}
&  \sum_{\al_{k}\in\Z} \frac{1}{\max\big\{|\wt \al_k|, ((2k+1)K)^{4p}\big\}^{p'}}\\
 &\quad \leq  \sum_{\substack{\al_{k}\in\Z\\|\al_k+\wt \al_{k-1}|\geq ((2k+1)K)^{4p}}}
  \frac{1}{|\al_k+\wt \al_{k-1}|^{p'}}\\
 &\hphantom{XXXXXX}  + ((2k+1)K)^{-4p p'}\sum_{\substack{\al_{k}\in\Z\\
 |\al_k  + \wt \al_{k-1}| <  ((2k+1)K)^{4p}}}1\\
 &\quad =  \sum_{|n|\geq ((2k+1)K)^{4p}}\frac{1}{|n|^{p'}}+ ((2k+1)K)^{4p(1-p')}\\
 & \quad \le B_p ((2k+1)K)^{- 4p'}
\end{split}
\label{eta7}
\end{align}

\noi
with the understanding that $\wt \al_0 = 0$, 
where $B_p>0$ is a constant depending only on $p$. 
We note that, in the last step, we used the fact that $p < \infty$ (and thus $p' > 1$).
Then, by summing
 in the order $\al_{j-1}, \al_{j-2},\dots, \al_1$
 with~\eqref{eta7}, we obtain
\begin{align}
\begin{split}
\S^{(j)}
 &\leq B^{j-1}_p \prod_{k=1}^{j-1}((2k+1)K)^{-4p'}\\
 &=B^{j-1}_pK^{4p'(1-j)}((2j-1)!!)^{-4p'}.
\end{split}
\label{eta8}
\end{align}

\medskip

It remains to bound $\K^{(j)}(\TT_{j-1}; n)$ in \eqref{eta6}.
We claim that the following bound holds:
\begin{align}\label{eta9}
  \sup_{\substack{\TT_{j-1}\in \mathfrak{T}(j-1)\\ n\in\Z^2}}
\K^{(j)}(\TT_{j-1}; n)
\le C^{j-1}, 
\end{align}

\noi
provided that 
 $s>1-\frac{1}{p}$. 
Then, the desired bound \eqref{M0j2} follows from 
\eqref{M0j},  
 \eqref{eta5}, 
 \eqref{eta8}, and~\eqref{eta9}.

We prove \eqref{eta9} by induction.
We first consider the $j = 2$ case.
Note that there is only one tree
$\TT_{1}$ in $\mathfrak{T}(1)$.
Then, 
from \eqref{eta6}, we have
\begin{align}
\begin{split}
\K^{(2)}(\TT_{1}; n)
& = 
\sup_{\al_1\in\Z}\sum_{\substack{\mathbf{n}\in \mathfrak{N}(\TT_{1})\\ \mathbf{n}_r=n\\ 
\mu_1 = \al_1}}\frac{\langle n\rangle^{sp'}}{\prod_{a\in \TT^{\infty}_{1}}\langle n_a\rangle^{sp'}}\\
& = 
\sup_{\al_1\in\Z}\sum_{(n_1, n_2, n_3) \in \G_{\al_1}^-(n)}
\frac{\langle n\rangle^{sp'}}{\prod_{j = 1}^3 \jb{n_j}^{sp'}}, 
\end{split}
\label{eta9a}
\end{align}

\noi
where
$\G_{\al_1}^-(n)$ is as in \eqref{GG0}.
By symmetry, assume $|n_1| \ge |n_2|, |n_3|$,
which implies $|n_1|\ges |n|$.
Then, from Lemma~\ref{LEM:count1}, we have 
\begin{align*}
\K^{(2)}(\TT_{1}; n)
& \les  
\sup_{\al_1\in\Z}\sum_{(n_1, n_2, n_3) \in \G_{\al_1}^-(n)}
\frac{1}{ \jb{n_2}^{sp'} \jb{n_3}^{sp'}}
\les 1,  
\end{align*}

\noi
uniformly in $n \in \Z^2$, 
provided that $sp' > 1$, namely, $s > 1 - \frac 1p$.
This proves \eqref{eta9} when $j = 2$.

Now, suppose that the bound \eqref{eta9} holds for given $j \ge 2$.
We prove \eqref{eta9} for $j + 1$.
Namely, we prove
\begin{align}\label{eta11}
  \sup_{\substack{\TT_{j}\in \mathfrak{T}(j)\\ n\in\Z^2}}
\K^{(j+1)}(\TT_{j}; n)
\le C^{j}. 
\end{align}

\noi
Fix $\TT_{j}\in \mathfrak{T}(j)$
and 
let $\TT_{j - 1} = \pi_{j - 1}(\TT_j)$, where  $\pi_{j-1}$ is as in \eqref{cj2}.
In obtaining the ordered $\TT_j$, we replaced one of the terminal nodes, say 
$b \in \TT_{j - 1}^\infty$,  into a nonterminal node.
In particular, $a \in \TT_j^\infty \setminus \TT_{j - 1}^\infty$
must be a child of $b$.
Then, with this notation, we have 
\begin{align}
 \sum_{\substack{\mathbf{n}\in \mathfrak{N}(\TT_{j})\\ \mathbf{n}_r=n}}
 = \sum_{\substack{\mathbf{n}\in \mathfrak{N}(\TT_{j-1})\\ \mathbf{n}_r=n}}
\sum_{\substack{n_{b}=n_1^{(j)}-n_2^{(j)}+n_3^{(j)} \\
n_{b}\ne n_1^{(j)}, n_3^{(j)}
}}.  
\label{eta12}
\end{align}

\noi
Hence, from \eqref{eta6} with \eqref{eta12}, we have 
\begin{align}
\begin{split}
\K^{(j+1)}(\TT_{j}; n)
& = 
\sup_{\al_1,\dots,\al_{j}\in\Z}\sum_{\substack{\mathbf{n}\in \mathfrak{N}(\TT_{j})\\ \mathbf{n}_r=n\\ 
\mu_k = \al_k, k = 1, \dots, j }}\frac{\langle n\rangle^{sp'}}{\prod_{a\in \TT^{\infty}_{j}}\langle n_a\rangle^{sp'}}\\
& \le  \sup_{\al_1,\dots,\al_{j-1}\in\Z}\sum_{\substack{\mathbf{n}\in \mathfrak{N}(\TT_{j-1})\\ \mathbf{n}_r=n\\ 
\mu_k = \al_k, k = 1, \dots, j - 1}}\frac{\langle n\rangle^{sp'}}{\prod_{a\in \TT^{\infty}_{j-1}\setminus \{b\}}
\jb{n_a}^{sp'}\cdot \jb{n_b}^{sp'}}\\
& \quad \times \sup_{\al_j\in\Z}\sum_{(n_1^{(j)}, n_2^{(j)}, n_3^{(j)}) \in \G_{\al_j}^-(n_b)}
\frac{\jb{n_b}^{sp'}}{ \jb{n_1^{(j)}}^{sp'}\jb{n_2^{(j)}}^{sp'} \jb{n_3^{(j)}}^{sp'}}, 
\end{split}
\label{eta13}
\end{align}

\noi
where $\G_{\al_j}^-(n_b)$ is as in \eqref{GG0}.
We can now bound the first factor on the right-hand side of~\eqref{eta13}
by the inductive hypothesis, 
while we bound the second
factor on the right-hand side of \eqref{eta13}
by proceeding as in the $j = 2$ case
(compare \eqref{eta9a} with the second
factor on the right-hand side of \eqref{eta13}).
This yields \eqref{eta11}.
Hence, by induction, we
conclude \eqref{eta9} for any integer $j \ge 2$.
This concludes
the proof of Lemma \ref{LEM:N0j}.
\end{proof}

\begin{remark}\label{REM:diff1}\rm

As compared to  the proof of \cite[Lemma 3.12]{OW2} in the one-dimensional elliptic case, 
the structure of the 
 proof of  Lemma \ref{LEM:N0j} presented above
is slightly different. 
In the proof of \cite[Lemma 3.12]{OW2}, 
an expression analogous to \eqref{M0j}
was bounded directly via the divisor counting argument. 
On the other hand, 
in the  proof of  Lemma \ref{LEM:N0j}, 
we split
$\M^{(j)}_0(\TT_{j-1}; n)$ in~\eqref{M0j} into two parts (see \eqref{eta5})
and estimated each part separately.
In particular,
we employed the hyperbolic counting estimate (Lemma \ref{LEM:count1})
and an induction 
in establishing the bound~\eqref{eta9}.
Note that the restriction $s > 1- \frac 1p$
comes precisely from the 
 use of 
 the hyperbolic counting estimate (Lemma \ref{LEM:count1})
 which requires $sp' > 1$.

\end{remark}

\begin{remark}\label{REM:diff2}\rm 

Following the convention in 
Remark \ref{REM:con1}, 
the bound \eqref{N0x} 
in Lemma \ref{LEM:N0j}
is stated
only for the case when all argument coincide.
In proving Theorem \ref{THM:3}\,(i)
via a contraction argument as in 
\cite[Subsection 2.1]{OW2}, 
we also need to establish a different estimate:
 \begin{align}
 \begin{split}
&  \| \NN _0^{(j)}(\uu  ) - \NN _0^{(j)}(\vv  )\|_{\mathcal{F}L^{s,p}(\T^2)}\\
 &\quad  \leq 
 C_p\frac{K^{4(1-j)}}{((2j-1)!!)^{2}}
 \Big(\| \uu  \|^{2j-2}_{\mathcal{F}L^{s,p}(\T^2)}
+  \| \vv  \|^{2j-2}_{\mathcal{F}L^{s,p}(\T^2)}\Big)
\| \uu - \vv \|_{\mathcal{F}L^{s,p}(\T^2)}.
\end{split}
\label{N0y}
\end{align}

\noi
By the multilinearity of $\NN _0^{(j)}$, 
we can write 
$\NN _0^{(j)}(\uu  ) - \NN _0^{(j)}(\vv  )$
as the sum of $O(J)$ differences, 
each of which contains exactly one factor of $\uu - \vv$;
see the proof of \cite[Lemma 3.11]{GKO}.
This $O(J)$ loss does not affect the bound \eqref{N0y}
thanks to the fact decay \eqref{eta8}.
The same  comment applies to the other multilinear operators
studied in the next subsection.

\end{remark}

\subsection{Other multilinear operators}
\label{SUBSEC:NF2b}

We first consider  the term $\RR ^{(j)}(\uu  )$ in \eqref{NN2j}
(with $j + 1$ replaced by $j$).
Given an integer $j \ge 2$, we have 
\begin{align}
\begin{split}
& \RR ^{(j)}(\uu  )(t, n)\\
& \quad =\sum_{\TT_{j-1}\in \mathfrak{T}(j-1)} 
\sum_{\substack{\mathbf{n}\in \mathfrak{N}(\TT_{j-1})\\\mathbf{n}_r=n}} 
\sum_{b\in\TT_{j-1}^{\infty}}
\ind_{\bigcap_{k=1}^{j-1}\text{A}_k^c}\frac{e^{i t \widetilde{\mu}_{j-1}}}{\prod_{k=1}^{j-1}\widetilde{\mu}_k} \RR ^{(1)}(\uu  )(n_b)\prod_{a\in \TT^{\infty}_{j-1}\setminus \{b\}}\widehat{\uu  }_{n_a}, 
\end{split}
\label{Rj}
\end{align}

\noi
where $\RR ^{(1)}$
is as in \eqref{HNLS2}.
See Lemma \ref{LEM:R1} for the $j = 1$ case.

\begin{proof}[Proof of Lemma \ref{LEM:Rj}]
The bound 
\eqref{Rjx}
follows from \eqref{Rj}, Lemma \ref{LEM:N0j}, and Lemma \ref{LEM:R1}
with
$|\TT_{j-1}^\infty| = 2j - 1$.
\end{proof}

Next, we consider the nearly resonant term 
$    \NN ^{(j)}_1(\uu  )$ in \eqref{NN2ja}
(with $j + 1$ replaced by $j$; see also \eqref{NN2j}), 
 given by 
\begin{align*}
    \NN ^{(j)}_1(\uu  )(t, n)=\sum_{\TT_{j}\in \mathfrak{T}(j)} \sum_{\substack{\mathbf{n}\in \mathfrak{N}(\TT_{j})\\\mathbf{n}_r=n}} \ind_{(\bigcap_{k=1}^{j-1}\text{A}_k^c) \cap
 A_j}\frac{e^{it \widetilde{\mu}_{j}}}{\prod_{k=1}^{j-1}\widetilde{\mu}_k} \prod_{a\in \TT^{\infty}_{j}}\widehat{\uu  }_{n_a}.
\end{align*}

\begin{proof}[Proof of Lemma \ref{LEM:N1j}]

We proceed as in the proof of Lemma \ref{LEM:N0j}.
By H\"older's inequality  with \eqref{cj1}, we obtain
\begin{align*}
\| \NN _1^{(j)}(\uu  )(t)\|_{\mathcal{F}L^{s,p}}
 &\leq \M_1^{(j)}
\sum_{\TT_{j}\in \mathfrak{T}(j)}
\bigg\| \bigg(\sum_{\substack{\mathbf{n}\in \mathfrak{N}(\TT_{j})\\\mathbf{n}_r=n}}\prod_{a\in \TT^{\infty}_{j}}
\langle n_a \rangle^{sp}|\widehat{\uu  }_{n_a}|^p\bigg)^{\frac 1p} \bigg\|_{\ell^p_n}\\
& \leq (2j-1)!!
\cdot \M_1^{(j)}
\| \uu  \|^{2j+1}_{\mathcal{F}L^{s,p}}, 
\end{align*}

\noi
uniformly in $t \in \R$, 
where $\M^{(j)}_1$ is defined by 
\begin{align}
\begin{split}
\M^{(j)}_1
& = 
\sup_{\substack{\TT_{j}\in \mathfrak{T}(j)\\n\in\Z^2}}
\M^{(j)}_1(\TT_{j}; n)\\
:\!& = \sup_{\substack{\TT_{j}\in \mathfrak{T}(j)\\n\in\Z^2}}
\bigg(\sum_{\substack{\mathbf{n}\in \mathfrak{N}(\TT_{j})\\ \mathbf{n}_r=n}} \frac{\langle n\rangle^{sp'}}{\prod_{a\in \TT^{\infty}_{j}}\langle n_a\rangle^{sp'}}\frac{\ind_{(\bigcap_{k=1}^{j-1}\text{A}_k^c)\cap A_j}}{\prod_{k=1}^{j-1}|\widetilde{\mu}_k|^{p'}}\bigg)^{\frac 1{p'}}.
\end{split}
\label{M1j}
\end{align}

\noi
Hence, it suffices to  prove
\begin{align}
\M^{(j)}_1 \les
C_p\frac{K^{4(p-j)}}{((2j-1)!!))^{2}}.
\label{M1j2}
\end{align}

With the notations as in the proof of Lemma \ref{LEM:N0j}, 
it follows from 
 H\"older's inequality with  \eqref{M1j}, 
 \eqref{eta4}, 
and 
\eqref{Aj} that 
\begin{equation}\label{fta5}
\begin{aligned}
\big(\M^{(j)}_1(\TT_{j}; n)\big)^{p'}
&\le 
\wt \S^{(j)}\cdot \K^{(j+1)}(\TT_{j}; n), 
\end{aligned}
\end{equation}

\noi
uniformly in 
$\TT_{j}\in \mathfrak{T}(j)$ and 
$n \in \Z^2$, 
where 
$\K^{(j+1)}(\TT_{j}; n)$ is as in \eqref{eta6}
and 
$\wt \S^{(j)}$ is   defined by 
\begin{align*}
\wt \S^{(j)}
& = \sum_{\al_1,\dots,\al_{j}\in \Z}
\frac{\ind_{| \al_j + \wt \al_{j - 1}| \le ((2 j + 1)K)^{4p}}  }{\prod_{k=1}^{j-1}\max\big\{|\wt \al _k|, ((2k+1)K)^{4p}\big\}^{p'}}.
\end{align*}

As compared to  $\S^{(j)}$ in \eqref{eta6}, 
we have an additional summation in $\al_j$
under the restriction 
$| \al_j + \wt \al_{j - 1}| \le ((2 j + 1)K)^{4p}$.
We  first sum in $\al_j$ (with $\wt \al_{j - 1}$ fixed)
and then 
 sum
 in the order $\al_{j-1}, \al_{j-2},\dots, \al_1$
 with~\eqref{eta7}.
 Then, from~\eqref{eta8}, we obtain
\begin{align}
\begin{split}
\wt \S^{(j)}
 &\le  B^{j-1}_p ((2 j + 1)K)^{4p} \cdot K^{4p'(1-j)}((2j-1)!!)^{-4p'}\\
 &=   B^{j-1}_p ((2 j + 1))^{4p} \cdot K^{4p'(p-j)}((2j-1)!!)^{-4p'}\\
 &\les    B^{j-1}_p  \cdot K^{4p'(p-j)}((2j-1)!!)^{-3p'}, 
\end{split}
\label{fta8}
\end{align}

\noi
uniformly in $j \in \NB$, 
provided that $s > 1- \frac 1p$.
Hence, the desired bound \eqref{M1j2}
follows from~\eqref{fta5}, 
\eqref{eta9}, and 
\eqref{fta8}.
This proves \eqref{Njx}.
\end{proof}

\subsection{On the error term}
\label{SUBSEC:NF2c}

We conclude this section by presenting a proof of 
 Lemma \ref{LEM:N2j}.

\begin{proof}[Proof of Lemma \ref{LEM:N2j}]

Recall that from \eqref{NN2ja}, we write 
\begin{align}\label{ER0}
    \NN _2^{(J)}(\uu  )=\NN ^{(J)}(\uu  )-\NN _1^{(J)}(\uu  ).
\end{align}

\noi
In the following, we estimate $\NN ^{(J)}(\uu  )$.

From \eqref{HNLS2}, 
H\"older's inequality on the physical side or 
 Young's inequality on the Fourier side with the embeddings~\eqref{emb0} and \eqref{emb1},   
we have   
\begin{align}\label{ER1}
\| \NN ^{(1)}(\uu  ) (t) \|_{\mathcal{F}L^{\infty}}
\le
\min \big(\|S(t) \uu\|_{L^3}^3, \| \uu  \|^3_{\mathcal{F}L^{\frac{3}{2}}}\big)
\lesssim \| \uu  \|^3_{\mathcal{F}L^{s,p}}, 
  \end{align}

\noi
provided that \eqref{emb0a} or \eqref{emb2} holds.

From \eqref{NN2j}, we have
  \begin{align*}
\NN ^{(J)}(\uu  )(t, n)
& = \sum_{\TT_{J-1}\in \mathfrak{T}(J-1)} \sum_{\substack{\mathbf{n}\in \mathfrak{N}(\TT_{J-1})\\\mathbf{n}_r=n}}\sum_{b\in\TT^{\infty}_{J-1}} 
\ind_{\bigcap_{k=1}^{J-1}\text{A}_k^c}\\
& \hphantom{XXXX} \times \frac{e^{i t \widetilde{\mu}_{J-1}}}{\prod_{k=1}^{J-1}\widetilde{\mu}_k} 
\NN ^{(1)}(\uu  )(n_b)\prod_{a\in \TT^{\infty}_{J-1}\setminus\{b\}}\widehat{\uu  }_{n_a}  .
  \end{align*}

\noi
Then, proceeding as in \eqref{N0j1}
with H\"older's inequality, $|\TT^\infty_{J-1}| = 2J- 1$, 
\eqref{cj1}, and \eqref{ER1}
we have
\begin{equation}\label{ER2}
  \begin{aligned}
 \| \NN^{(J)}(\uu  )(t)\|_{\mathcal{F}L^{\infty}}
 &\les  \M^{(J)}
\sum_{\TT_{J-1}\in \mathfrak{T}(J-1)}
\sum_{b\in\TT^{\infty}_{J-1}} 
\bigg\| \bigg(\sum_{\substack{\mathbf{n}\in \mathfrak{N}(\TT_{J-1})\\\mathbf{n}_r=n}}
|\NN ^{(1)}(\uu  )(n_b)|\\
& \hphantom{XXXXXXXXXXXX}\times \prod_{a\in \TT^{\infty}_{J-1}\setminus \{b\} }\langle n_a \rangle^{sp}|\widehat{\uu  }_{n_a}|^p\bigg)^{\frac 1p} \bigg\|_{\ell^\infty_n}\\
& \leq  (2J-1)!!
\cdot \M^{(J)}
\| \uu  \|^{2J+1}_{\mathcal{F}L^{s,p}}, 
\end{aligned}  
\end{equation}

\noi
uniformly in $t \in \R$, 
where $\M^{(J)}$ is defined by 
\begin{align}
\begin{split}
\M^{(J)}
& = 
\sup_{\substack{\TT_{J-1}\in \mathfrak{T}(J-1)\\
b\in\TT^{\infty}_{J-1}\\
n\in\Z^2}}
\M^{(J)}(\TT_{J-1}; b, n)\\
:\!& = \sup_{\substack{\TT_{J-1}\in \mathfrak{T}(J-1)\\
b\in\TT^{\infty}_{J-1}\\
n\in\Z^2}}
\bigg(\sum_{\substack{\mathbf{n}\in \mathfrak{N}(\TT_{J-1})\\ \mathbf{n}_r=n}} 
\frac{1}{\prod_{a\in \TT^{\infty}_{J-1}\setminus \{b\}}\langle n_a\rangle^{sp'}}
\frac{\ind_{\bigcap_{k=1}^{J-1}\text{A}_k^c}}{\prod_{k=1}^{J-1}|\widetilde{\mu}_k|^{p'}}\bigg)^{\frac 1{p'}}.
\end{split}
\label{ER3}
\end{align}

By H\"older's inequality with  \eqref{ER3}
and \eqref{eta4}, we have 
\begin{equation}\label{ER4}
\begin{aligned}
\big(\M^{(J)}(\TT_{J-1}; b, n)\big)^{p'}
&\le 
\S^{(J)}\cdot \wt \K^{(J)}(\TT_{J-1}; b, n), 
\end{aligned}
\end{equation}

\noi
uniformly in 
$\TT_{J-1}\in \mathfrak{T}(J-1)$, 
$b\in\TT^{\infty}_{J-1}$,  and 
$n \in \Z^2$, 
where 
$\S^{(J)}$ is as in \eqref{eta6}
and  $\wt \K^{(J)}(\TT_{J-1}; b, n)$ is  defined by 
\begin{align*}
\wt \K^{(J)}(\TT_{J-1}; b, n)
& = 
\sup_{\al_1,\dots,\al_{J-1}\in\Z}\sum_{\substack{\mathbf{n}\in \mathfrak{N}(\TT_{J-1})\\ \mathbf{n}_r=n\\ 
\mu_k = \al_k, k = 1, \dots, J - 1}}\frac{1}
{\prod_{a\in \TT^{\infty}_{J-1} \setminus \{b\}}\langle n_a\rangle^{sp'}}.
\end{align*}

Proceeding inductively 
as in the proof of Lemma \ref{LEM:N0j}, 
we obtain the following bound:
\begin{align}\label{ER5}
  \sup_{\substack{\TT_{J-1}\in \mathfrak{T}(J-1)\\ 
b\in\TT^{\infty}_{J-1}\\
n\in\Z^2}}
\wt \K^{(J)}(\TT_{J-1}; b, n)
\le C^{J-1}.
\end{align}


\noi
Hence, from \eqref{eta8} and \eqref{ER5}, we have 
\begin{align}
\M^{(J)} \les
C_p^j\frac{K^{4(1-J)}}{((2J-1)!!)^4}
\label{ER6}
\end{align}

\noi
for some 
 constant 
 $C_p>0$ 
 depending only on $p$.

Therefore, 
putting \eqref{ER0},
\eqref{ER2}, \eqref{ER4}, and \eqref{ER6} with  \eqref{Njx} in   Lemma~\ref{LEM:N1j}, 
we obtain
\eqref{N2jx}.
\end{proof}

 \appendix

\section{Proof of Theorem \ref{THM:1}}
\label{SEC:A}

In this appendix, we present a proof of Theorem \ref{THM:1}.
In Subsection \ref{SUBSEC:A2}, 
we also prove Theorem \ref{THM:3}\,(ii).

\subsection{Local well-posedness via the Fourier restriction norm method}
\label{SUBSEC:A1}

In this subsection, 
we present a proof of Theorem \ref{THM:1}\,(i).

Given $s, b\in\mathbb{R}$ and $1 <  p <  \infty$,
define the $X^{s,b}_p$-space via the norm (see \cite{GH, FOW, C1, C2}):
\begin{align}
\| u\|_{X^{s,b}_p}=\| \jb{n}^s\jb{\tau+|n|_-^2}^b\ft u (\tau, n)\|_{\ell^{p}_nL^p_{\tau}(\Z^2\times \mathbb{R})}, 
\label{Xsb1}
\end{align}

\noi
where $|n|_-^2$ is the hyperbolic modulus defined  in \eqref{GG0a}.
For $b p'> 1$, we have 
\begin{align}
 X^{s, b}_{p}\subset C(\mathbb{R}; \mathcal{F}L^{s,p}(\T^2)).
 \label{Xsb1a}
\end{align}

\noi
Given $T > 0$, 
we define the local-in-time version of the $X^{s, b}_p$-space
by setting
 \begin{align}
\| u\|_{X^{s,b}_{p}(T) }=\inf\big\{\| v\|_{X^{s,b}_{p}}: v|_{[0,T]}=u\big\}, 
\label{Xsb2}
\end{align}
     
\noi
where the infimum is taken over all extensions $v$ of $u$
from $[0, T]$ to $\R$.

We first recall the following linear estimates; 
see \cite[Lemma 2.2]{FOW} for a proof.

\begin{lemma}\label{LEM:non0}

\textup{(i)
(homogeneous linear estimate).}
 Let  $s, b\in\mathbb{R}$
 and $1 <  p <  \infty$.
Then,  we have
\begin{align*}
\| S(t)f\|_{X^{s,b}_{p}(T)} \lesssim \| f\|_{\mathcal{F}L^{s,p}}     
\end{align*}

\noi
 for any $0<T\leq 1$.

\medskip

\noi
\textup{(ii)
(nonhomogeneous linear estimate).} 
Let  $s\in\mathbb{R}$, $1 <  p< \infty$, and $-\frac{1}{p}<b'\leq 0 \leq b \leq 1+b'$. Then, we have
\begin{align*}
\bigg\| \int_0^tS(t-t')F(t')dt'\bigg\|_{X^{s,b}_{p}(T)} \lesssim T^{1+b'-b}\| F\|_{X^{s,b'}_{p}(T)}     
\end{align*} 

\noi
for any $0<T\leq 1$.
      
\end{lemma}

Next, we state the key trilinear estimate.

\begin{proposition}\label{PROP:non1}
Let $1<  p <\infty$ and $s>1-\frac{1}{p}$.
 Then,  there exists small $\eps>0$ such that 
\begin{align*}
\| u_1 \cj u_2 u_3\|_{X^{s, -\frac{1}{p}+2\eps}_{p}(T)}\lesssim 
\prod_{j = 1}^3 \| u_j\|_{X^{s, 1-\frac{1}{p}+\eps}_{p}(T)}.
    \end{align*}

\noi
for any $0 < T \le 1$.
\end{proposition}

Once we prove Proposition \ref{PROP:non1}, 
Theorem \ref{THM:1}\,(i) follows from a standard contraction argument, 
applied to the Duhamel formulation \eqref{Duha1}, 
with Lemma \ref{LEM:non0} and Proposition \ref{PROP:non1}.
Since the argument is standard, we omit details.

\medskip

Before proceeding further, we  recall the following calculus lemma; 
 see, for example,  \cite[Lemma~4.2]{Ginibre}.

\begin{lemma}\label{LEM:non2}

Let $\beta\geq \gamma\geq 0$ and $\be + \gamma>1$. Then, we have 
\begin{align*}
\int_{\R}
\frac{d\tau_1}{\jb{\tau - \tau_1}^\be  \jb{\tau_1}^\g}
\lesssim \frac 1{\jb{\tau}^\al}
\end{align*}

\noi
where $\al$ is given 
by \begin{align*}
\al = 
\begin{cases}
\g, & \text{if } \be > 1, \\
\g- \eps, & \text{if } \be = 1, \\
\be + \g - 1, & \text{if } \be < 1 
\end{cases}
\end{align*}

\noi
for any $\eps > 0$.
\end{lemma}

We now present a proof of Proposition \ref{PROP:non1}.

 \begin{proof}[Proof of Proposition \ref{PROP:non1}]
By the definition of the restricted norm \eqref{Xsb2}, it suffices to prove 
the following bound:
\begin{align}
\| u_1 \cj u_2 u_3\|_{X^{s, -\frac{1}{p}+2\eps}_{p}}\lesssim 
\prod_{j = 1}^3 \| u_j\|_{X^{s, 1-\frac{1}{p}+\eps}_{p}}
\label{tri2}
\end{align}

\noi
without time localization.
By writing 
\begin{align}\label{tri1a}
u_1 \cj u_2 u_3
= \Nf(u_1, u_2, u_3) + \Rf_1(u_1, u_2, u_3) + \Rf_2(u_1, u_2, u_3), 
\end{align}

\noi
where $\Nf$, $\Rf_1$, and $\Rf_2$ are as in \eqref{non1}, 
we estimate
the $X^{s, -\frac{1}{p}+2\eps}_{p}$-norm of each term
on the right-hand side of \eqref{tri1a} in the following.

  Let $\s_0=\tau+|n|_-^2$ and 
  $\s_j=\tau_j+|n_j|_-^2$,  $j=1,2,3$. 
  The estimate \eqref{tri2} 
  for  the non-resonant part $\Nf(u_1, u_2, u_3)$ follows, once we prove
\begin{align}\label{tri3}
\bigg\| \sum_{\substack{n=n_1-n_2+n_3\\ n \neq n_1, n_3 }}
\, \intt_{\tau=\tau_1-\tau_2+\tau_3}M\prod_{j=1}^3f_j^*(\tau_j, n_j)d\tau_1d\tau_2
\bigg\|_{\ell^p_nL^p_{\tau}}\lesssim \prod_{j=1}^3\| f_j\|_{\ell^p_nL^p_{\tau}}, 
\end{align}

\noi
 where 
 \begin{align*}
 f^*(\tau_j , n_j)
= \begin{cases}
 f(\tau_j, n_j), & \text{when $j= 1, 3$,}\rule[-3mm]{0pt}{0pt} \\
\cj{ f(\tau_2, n_2)}, & \text{when $j= 2$, }
\end{cases}
\end{align*}

\noi
and 
 \begin{align*}
     M=M(n, n_1, n_2, n_3, \tau, \tau_1, \tau_2, \tau_3)=\frac{\langle n\rangle^s}{\langle \sigma_0\rangle^{\frac{1}{p}-2\eps}\prod_{j=1}^3\langle n_j\rangle^{s}\langle \sigma_j\rangle^{\frac{1}{p'}+\eps}}.
 \end{align*} 
 
 \noi
  By H\"older's inequality, we have 
 \begin{align*}
   \text{LHS of } \eqref{tri3}\lesssim 
\bigg( \sup_{n,\tau}\sum_{\substack{n=n_1-n_2+n_3\\ n \neq n_1, n_3 }}
\, \int\limits_{\tau=\tau_1-\tau_2+\tau_3}M^{p'}d\tau_1d\tau_2\bigg)^{\frac{1}{p'}}\prod_{j=1}^3\| f_j\|_{\ell^p_nL^p_{\tau}}.    
  \end{align*}

\noi
Hence, \eqref{tri3} follows once we show
\begin{align}
\sup_{n,\tau} I_{n,\tau}(M) 
:= \sup_{n,\tau}\sum_{\substack{n=n_1-n_2+n_3\\ n \neq n_1, n_3 }}
  \, \int\limits_{\tau=\tau_1-\tau_2+\tau_3}M^{p'}d\tau_1d\tau_2
  <\infty.   
\label{tri4}  
\end{align}

By applying Lemma \ref{LEM:non2} 
integrations to $\tau_1, \tau_2$ integrals, we have 
\begin{align*}
 I_{n,\tau}(M)\lesssim \frac{ \langle n\rangle^{sp'}}{\langle\sigma_0\rangle^{\frac{1}{p-1}-2\eps p'}}\sum_{\substack{n=n_1-n_2+n_3\\ n \neq n_1, n_3 }}\frac{1}{\langle \Phi(\bar n)-\sigma_0\rangle^{1+\eps p'}\prod_{j=1}^3\langle n_j\rangle^{sp'}},    
\end{align*}

\noi
where $\Phi(\bar n)$ is as in \eqref{phi1}. 
By symmetry, we assume $|n_1| \ge |n_2|, |n_3|$,
which implies $|n_1|\ges |n|$.
Then, from the hyperbolic counting estimate (Lemma 
\ref{LEM:count1}), we have, for any $\ta > 0$, 
\begin{align*}
    I_{n,\tau}(M)
&\lesssim  
\sum_{\mu\in\Z}\frac{1}{\jb{\mu-\sigma_0}^{1+\eps p'}}
\sum_{\substack{N_2, N_3\ge 1\\\text{dyadic}}}\frac{1}{N_2^{sp'}N_3^{sp'}}
\sum_{\substack{n_1, n_2, n_3\in \Z^2\\n=n_1-n_2+n_3
\\n_2\neq n_1, n_3 \\ |n_2|\sim N_2,|n_3|\sim N_3}}
\ind_{\substack{ \Phi(\bar n) = \mu }}\\
&\les
\sum_{\mu\in\Z}\frac{1}{\jb{\mu-\sigma_0}^{1+\eps p'}}
\cdot 
\sum_{\substack{N_2, N_3\ge 1\\\text{dyadic}}}\frac{1}{N_2^{sp'-1 - \ta}N_3^{sp'-1-\ta}}\\
& \les 1
\end{align*}

\noi
uniformly in $\tau \in \R$ and $n \in \Z^2$, 
provided that $s>\frac 1{p'} = 1-\frac{1}{p}$. 
(and by taking $\ta > 0$ sufficiently small such that $sp' > 1 + \ta$).
This proves~\eqref{tri4}.

Next, we consider $\Rf_1$ in \eqref{tri1a}.
From  \eqref{non1}, 
Young's inequality,  the embedding $\ell_n^p \subset \ell_n^{3p}$, 
and H\"older's inequality in $\tau$, 
we have 
\begin{align*}
\| \Rf_1(u_1, u_2, u_3)\|_{X^{s, -\frac{1}{p}+2\eps}_{p}}
&\le \big\| \jb{n}^s\big[\ft u_1(\cdot, n)* \cj{\ft  u_2 (\cdot, n)}*\ft u_3(\cdot, n)\big](\tau)\big\|_{\ell^p_nL^p_{\tau}} \\
&\le \prod_{j = 1}^3  \| \langle n \rangle^s\widehat u_j(n, \tau)\big\|_{\ell^{3p}_nL^\frac{3p}{2p+1}_{\tau}}\\
&\lesssim \prod_{j = 1}^3 \| u_j\|_{X^{s, 1 - \frac 1p + \eps}_p}^3, 
\end{align*}

\noi
where the convolutions are in temporal frequencies.
This proves  \eqref{tri2}
for $\Rf_1$. 

Lastly,  we consider $\Rf_2$ in \eqref{tri1a}.
From  \eqref{non1}, 
\eqref{crit2}, and \eqref{Xsb1a}, we have 
\begin{align*}
 \| \Rf_2(u_1, u_2, u_3)\|_{X^{s, -\frac{1}{p}+2\eps}_{p}}
&\le
\prod_{j = 1}^2 
\| u_j \|_{C_t L^2_x}\cdot \| u_3\|_{X^{s, 0}_{p}}\\
&\lesssim \prod_{j = 1}^3 \| u_j\|_{X^{s, 1 - \frac 1p + \eps}_p}^3, 
\end{align*}

\noi
yielding \eqref{tri2}
for $\Rf_2$. 
This concludes the proof of Proposition \ref{PROP:non1}.
 \end{proof}

\subsection{Failure of $C^3$-smoothness}
\label{SUBSEC:A2}

We conclude this paper by proving 
 Theorem~\ref{THM:1}\,(ii)
 and 
  Theorem~\ref{THM:3}\,(ii).
Proceeding as in \cite{BO97}, 
 Theorem~\ref{THM:1}\,(ii)
 on the failure of $C^3$-smoothness
of the solution map
for the cubic HNLS \eqref{HNLS1}
follows once we show unboundedness
of the Picard second iterate:
\begin{align}
 A[f](t):=i \int^t_0 
 S(t - t') \big(|S(t') f|^2 S(t') f\big) dt'.
\label{ill1}
 \end{align}

\noi
In \cite{Wang}, 
the third author treated the 
 $p =2$ case 
by exploiting the observation 
that  
\[S(t) f = f\]
 
\noi
for any $t \in \R$, 
if  the Fourier support of $f$
is contained in the diagonal $\Dl(\Z^2)$  of $\Z^2$, 
namely, 
\begin{align}
\supp \ft f \subset \Dl(\Z^2): = \{ (k, k) \in \Z^2: k \in \Z\}.
\label{ill2}
\end{align}

\noi
We exploit the same idea
and prove the following lemma,
from which Theorem \ref{THM:1}\,(ii) follows.

 \begin{lemma}\label{LEM:ill}
Let $s \in \R$ and $1\leq p <\infty$.
Given  $N\in \NB$, 
define $f_N \in C^\infty(\T^2)$
by 
\begin{align}
f_N=N^{- s-\frac{1}{p}}\sum_{\substack{n\in \Dl(\Z^{2})\\ |n| \le N}}e^{in \cdot x}
\label{ill3}
\end{align}

\noi
such that $\|f_N \|_{\FL^{s, p}} \sim 1$, 
where $\Dl(\Z^2)$ is as in \eqref{ill2}.
Then,  given any $t > 0$, we have
\begin{align}\label{ill4}
\| A[f_N](t)\|_{\mathcal{F}L^{s,p}(\T^2)}\gtrsim t N^{2-2s-\frac{2}{p}},   
\end{align}

\noi
where $A$ is as in \eqref{ill1}.
In particular, the Picard second iterate operator $A$ in \eqref{ill1}
is unbounded on $\FL^{s, p}(\T^2)$ if $s <  1- \frac 1p$.

 \end{lemma}

 \begin{proof} 
 Fix $t > 0$. Since $\supp \ft f_N \subset \Dl(\Z^2)$, 
 we have $S(t) f_N = f_N$.
 Then, from~\eqref{ill1}, 
 we have 
 \begin{align*}
\ft {A[f_N]}(t,n)= i t N^{-3s -\frac{3}{p}}
\sum_{\substack{n_1, n_2, n_3 \in \Dl(\Z^2)\\
|n_1|, |n_2|, |n_3| \sim N}}
\ind_{n = n_1 - n_2 + n_3}
 \end{align*}

\noi
and thus 
 \begin{align*}
|\ft {A[f_N]}(t,n)|
\sim   t N^{2-3s -\frac{3}{p}}, 
 \end{align*}

\noi
from which \eqref{ill4} follows.
 \end{proof}

Let us now turn to a proof of 
Theorem \ref{THM:3}\,(ii).
We proceed as in \cite{BO97}, 
but in the context of the normal form equation \eqref{NF1a}.
Given $N \in \NB$, let $f_N$ be as in \eqref{ill3}.
Then, given $\dl > 0$, 
let $u_{\dl, N}$ be the (unique) smooth solution
to the normal form equation  \eqref{NF1a}
with $u_{\dl, N}|_{t = 0} = \dl f_N$.
Proceeding as in \cite{BO97} with 
$u_{\dl, N}|_{\dl = 0} \equiv 0$, 
we see that the only non-zero contribution to 
 \begin{align*}
\frac {\dd^3 u_{\dl, N} (t)}{\dd \dl^3} \bigg|_{\dl = 0}
\end{align*}

\noi
comes from 
the terms which are cubic in $S(t) f_N$.
Since  $\ft f_N$ is supported on 
the diagonal  $\Dl(\Z^2)$  of $\Z^2$
 (in particular, we have $|n|_- = 0$ for any $n \in \supp (\ft f_N$)), 
such cubic terms must be all resonant (i.e.~$\Phi(\bar n) = 0$).
Thus, from 
\eqref{NF1a}
with \eqref{HNLS2},  \eqref{NN1},  and \eqref{NF1b}, 
we have 
 \begin{align*}
\frac {\dd^3 u_{\dl, N} (t)}{\dd \dl^3} \bigg|_{\dl = 0}
& 
= \int_0^t S(t - t ') \bigg\{
\Nf_1^{(1)}(S(t') f_N)   +  \Rf^{(1)}(S(t') f_N)\bigg\} dt'\\
& = A[f_N](t).
\end{align*}

\noi
Then, the claim follows from Lemma \ref{LEM:ill}.

\begin{ack}	
\rm

E.B. and Y.W. were supported by the EPSRC New Investigator Award (grant no. EP/V003178/1).
T.O.~was supported by the European Research Council (grant no.~864138 ``SingStochDispDyn")
and by the EPSRC
Mathematical Sciences Small Grant (grant no. EP/Y033507/1).

\end{ack}

\end{document}